\documentclass[10pt, a4paper]{amsart}
\usepackage{amssymb}
\usepackage{amsmath}
\usepackage{amssymb}
\usepackage{amsthm}
\usepackage{bm}
\usepackage{bbm}
\usepackage{mathrsfs}
\usepackage{tikz}
\usetikzlibrary{decorations.markings}
\usetikzlibrary{decorations.pathmorphing}
\usepackage{xcolor}
\colorlet{ColorPink}{red!30}

\usepackage{enumitem}
\setenumerate{label={\rm (\alph{*})}}
\usepackage[colorlinks=true,linkcolor=blue,urlcolor=blue]{hyperref}
\usepackage{graphicx}

\numberwithin{equation}{section}

\newtheorem{theorem}{Theorem}[section]
\newtheorem{lemma}[theorem]{Lemma}
\newtheorem{proposition}[theorem]{Proposition}

\newtheorem{corollary}[theorem]{Corollary}

\newtheorem{definition}[theorem]{Definition}

\numberwithin{equation}{section}

\theoremstyle{remark}
\newtheorem{remark}[theorem]{Remark}

\normalsize

\def\Xint#1{\mathchoice
{\XXint\displaystyle\textstyle{#1}}%
{\XXint\textstyle\scriptstyle{#1}}%
{\XXint\scriptstyle\scriptscriptstyle{#1}}%
{\XXint\scriptscriptstyle\scriptscriptstyle{#1}}%
\!\int}
\def\XXint#1#2#3{{\setbox0=\hbox{$#1{#2#3}{\int}$ }
\vcenter{\hbox{$#2#3$ }}\kern-.6\wd0}}

\def\dashint{\Xint-}

\newcommand{\bd}{\operatorname{BD}}
\newcommand{\bv}{\operatorname{BV}}

\newcommand{\ld}{\operatorname{LD}}

\newcommand{\di}{\operatorname{div}}

\newcommand{\dif}{\operatorname{d}\!}

\newcommand{\spt}{\operatorname{spt}}

\newcommand{\R}{\mathbb{R}}

\newcommand{\locc}{\operatorname{loc}}

\renewcommand{\leq}{\leqslant}

\newcommand{\D}{\operatorname{D}\!}

\newcommand{\ball}{\operatorname{B}}
\newcommand{\dista}{\operatorname{dist}}

\newcommand{\sobo}{\operatorname{W}}

\newcommand{\lebe}{\operatorname{L}}

\newcommand{\hold}{\operatorname{C}}

\newcommand{\besov}{\operatorname{B}}

\newcommand{\dist}{\operatorname{dist}}

\newcommand{\E}{\operatorname{E}\!}

\newcommand{\sg}{\bm{\varepsilon}}
\newcommand{\sym}{\operatorname{sym}}
\newcommand{\trace}{\operatorname{Tr}}

\newcommand{\adop}{\bm{\varepsilon}}
\newcommand{\excesso}{\mathbf{U}}
\newcommand{\excenew}{\mathbf{E}}
\newcommand{\exce}{\widetilde{\mathbf{E}}}
\newcommand{\devi}{\operatorname{dev}}
\newcommand{\wstar}{\stackrel{*}{\rightharpoonup}}
\newcommand{\argmin}{\operatorname{argmin}}
\newcommand{\m}{\operatorname{m}}

\newcommand{\res}{\!\mathbin{\vrule height 1.6ex depth 0pt width
0.13ex\vrule height 0.13ex depth 0pt width 1.1ex}\!}

\usepackage{geometry}
\begin{document}
\title[Partial Regularity for Symmetric--Convex Functionals]{Partial Regularity for Symmetric--Convex Functionals of Linear Growth}
\author{Franz Gmeineder}
\address{University of Oxford \\ Mathematical Institute \\ Andrew Wiles Building }
\date{October 2016}
\maketitle
\hrulefill
\tableofcontents
\hrulefill
\section{Introduction}
Let $\Omega$ be an open and bounded Lipschitz subset of $\R^{n}$ with $n\geq 2$. In the present paper we study the partial regularity of minimisers of autonomous variational problems of the form
\begin{align}\label{eq:varprin}
\text{to minimise}\;\;\mathfrak{F}[u]:=\int_{\Omega}f(\sg(u))\dif x\;\text{over a prescribed Dirichlet class}\;\mathscr{D}. 
\end{align}
Here, $\sg(u)=\frac{1}{2}(\D u +\D^{\mathsf{T}}u)$ is the symmetric part of the gradient $\D u$, and $f\colon\R_{\sym}^{n\times n}\to\R$ is a convex $\hold^{2}$--variational integrand of linear growth, by which we shall understand that there exist two constants $0<c_{1}\leq c_{2}<\infty$ such that 
\begin{align}\label{eq:lingrowth}
c_{1}|\xi|\leq f(\xi)\leq c_{2}(1+|\xi|)\qquad\text{for all}\;\xi\in\R_{\sym}^{n\times n}. 
\end{align}
Since $f$ is a convex function defined on the symmetric matrices only, we shall refer to the functional $\mathfrak{F}$ as \emph{symmetric--convex}. Before we embark on the notion of minimisers and their partial regularity addressed in the main part of the paper, let us comment on the consequences of the growth condition \eqref{eq:lingrowth}. 

For this it is convenient to make a comparison with integrands enjoying better coercivity properties. So assume first that $f\in\hold^{2}(\R_{\sym}^{n\times n};\R_{\geq 0})$ is of $p$--growth with $1<p<\infty$, that is, there exist two constants $c_{1,p},c_{2,p}>0$ such that $c_{1,p}|\xi|\leq f(\xi)\leq c_{2,p}(1+|\xi|^{p})$ holds for all $\xi\in \R_{\sym}^{n\times n}$. Then $\mathfrak{F}$ is well--defined on $\sobo^{1,p}(\Omega;\R^{n})$, and after specifying a suitable Dirichlet class $\mathscr{D}=u_{0}+\sobo_{0}^{1,p}(\Omega;\R^{n})$, the boundedness of minimising sequences with respect to the Sobolev norm follows from \emph{Korn's inequality}. This inequality states that, given a bounded and open Lipschitz subset $\Omega$ of $\R^{n}$, for each $1<p<\infty$ there exists a constant $c=c(\Omega,p)>0$ such that 
\begin{align}\label{eq:korn}
\|u\|_{\sobo^{1,p}(\Omega;\R^{n\times n})}\leq c(\|u\|_{\lebe^{p}(\Omega;V)}+\|\sg(u)\|_{\lebe^{p}(\Omega;\R_{\sym}^{n\times n})})
\end{align} 
holds for all $u\in\sobo^{1,p}(\Omega;\R^{n})$. As a consequence, $\mathfrak{F}$--minimising sequences $(u_{k})$ are bounded in the reflexive Sobolev space $\sobo^{1,p}(\Omega;\R^{n})$ and thus, by the Banach--Alaoglu theorem, allow to extract subsequences converging weakly in $\sobo^{1,p}(\Omega;\R^{n})$. Since it easily can be shown that symmetric--convexity suffices to deduce lower semicontinuity of $\mathfrak{F}$  for weak convergence in $\sobo^{1,p}$, the weak limit is proved to be a minimiser of $\mathfrak{F}$. 

The deeper reason for a Korn--type inequality being available is that, given $\D$ is an elliptic differential operator, the mapping $\Phi\colon \hold_{c}(\R^{n};\R^{n\times n})\to \hold_{c}(\R^{n};\R_{\sym}^{n\times n})$ given by $\Phi\colon \sg(u)\mapsto \D u$ is a singular integral of convolution type. This point will be addressed in detail in section 2.1 below, but for the time being we confine to the well--established fact that singular integrals of convolution type are bounded as mappings $\lebe^{p}\to\lebe^{p}$ if and only if $1<p<\infty$. In the above setting of open Lipschitz subsets $\Omega$ of $\R^{n}$, the statement can be localised to yield the above version of Korn's inequality. However, if $p=1$, then it is well--known that singular integrals merely satisfy weak--(1,1)--type estimates, and this cannot be improved to obtain strong--(1,1)--type estimates. The unboundedness of singular integrals on $\lebe^{1}$ in turn implies that a Korn--type inequality \eqref{eq:korn} cannot hold for $p=1$. This negative result is often referred to as \emph{Ornstein's Non--Inequality}; see \cite{Ornstein} for \textsc{Ornstein}'s original work and \cite{KiKr,CFM,Gm1} for recent developments. 

\subsection{On the Notion of Minima}
As a straightforward implication of Ornstein's Non--Inequality, the space 
\begin{align}
\ld(\Omega):=\big\{v\in\lebe^{1}(\Omega;\R^{n})\colon\;\sg(v)\in\lebe^{1}(\Omega;\R^{n\times n})\big\}, 
\end{align}
where $\sg(v)$ is understood as the weak symmetric gradient of $v$, contains $\sobo^{1,1}(\Omega;\R^{n})$ as a proper subspace and is endowed with norm $\|v\|_{\ld}:=\|v\|_{\lebe^{1}}+\|\sg(v)\|_{\lebe^{1}}$. Since $\ld(\Omega)$ has trace space $\lebe^{1}(\partial\Omega;\R^{n})$ for open and bounded Lipschitz sets $\Omega\subset\R^{n}$, it is reasonable to define the Dirichlet $\mathscr{D}:=u_{0}+\ld_{0}(\Omega)$ for a given element $u_{0}\in\ld(\Omega)$, where $\ld_{0}(\Omega)$ is the $\ld$--norm closure of $\hold_{c}^{1}(\Omega;\R^{n})$. Although minimising sequences $(v_{k})\subset\mathscr{D}$ are then bounded in $\ld(\Omega)$ -- which differs from $\sobo^{1,1}$ by Ornstein's Non--Inequality -- non--reflexivity of the latter space does not allow to conclude weak sequential compactness from boundedness. Hence the variational problem \eqref{eq:varprin} must be lifted to the space $\bd(\Omega)$ of functions of bounded deformation. This space is defined to consist of all $v\in\lebe^{1}(\Omega;\R^{n})$ such that the distributional symmetric gradient can be represented by a $\R_{\sym}^{n\times n}$--valued Radon measure of finite total variation. In consequence, not being defined for measures, $\mathcal{F}$ must be relaxed to the space $\bd(\Omega)$. To do so, we shall consider the weak*--Lebesgue--Serrin extension given by 
\begin{align}
\overline{\mathfrak{F}}[v]:=\inf\left\{\liminf_{k\to\infty}\mathfrak{F}[v_{k}]\colon\;\begin{array}{c} (v_{k})\subset\mathscr{D},\\ v_{k}\stackrel{*}{\rightharpoonup}v\;\text{in}\;\bd(\Omega)\end{array}\right\}\qquad\text{for}\;v\in\bd(\Omega), 
\end{align}
where we say that a sequence $(v_{k})\subset\bd(\Omega)$ converges to $v\in\bd(\Omega)$ in the weak*--sense if and only if $v_{k}\to v$ strongly in $\lebe^{1}(\Omega;\R^{n})$ and $\sg(v_{k})\stackrel{*}{\rightharpoonup} \sg(v)$ in the sense of matrix--valued Radon measures as $k\to\infty$. Adopting this terminology, we make the following
\begin{definition}[$\bd$--Minimisers]
We say that a function $u\in\bd(\Omega)$ is a \emph{$\bd$--minimiser} of $\mathfrak{F}$ provided $\overline{\mathfrak{F}}[u]\leq\overline{\mathfrak{F}}[v]$ for all $v\in\bd(\Omega)$. 
\end{definition}
For a given open Lipschitz subset $\omega$ of $\Omega$, $u\in\bd(\Omega)$ and a function $v\in\bd(\omega)$, put
\begin{align*}
\widetilde{\mathfrak{F}}_{u}[v;\omega] :=\int_{\omega}f(\mathscr{E}v)\dif x + \int_{\omega}f^{\infty}\left(\frac{\dif \E v}{\dif |\E^{s}v|}\right)\dif |\E^{s}v| + \int_{\partial\omega}f^{\infty}(\nu_{\partial\omega}\odot\trace(v-u))\dif\mathcal{H}^{n-1}, 
\end{align*}
where 
\begin{align*}
\E v := \sg(v) & = \E^{ac}u+\E^{s}u \\
& = \mathscr{E}v\mathscr{L}^{n}\res\Omega + \frac{\dif \E v}{\dif |\E^{s}v|}|\E^{s}v|\res\Omega + (\nu_{\partial\Omega}\odot\trace(v))\mathcal{H}^{n-1}\res\partial\omega
\end{align*}
is the Radon--Nikod\v{y}m decomposition of the symmetric gradient measure $\E v$ and the map $f^{\infty}\colon\R_{\sym}^{n\times n}\to\R_{\geq 0}$ is the \emph{recession function} given by 
\begin{align}\label{eq:recession}
f^{\infty}(\xi):=\lim_{t\searrow 0}tf\left(\frac{\xi}{t}\right),\qquad \xi\in\R_{\sym}^{n\times n}. 
\end{align}
It then can be shown that a function $u\in\bd(\Omega)$ is a $\bd$--minimiser if and only if $\widetilde{\mathfrak{F}}_{u_{0}}[u]\leq \widetilde{\mathfrak{F}}_{u_{0}}[v]$ for all $v\in\bd(\Omega)$. Since our main result will be of local nature, we moreover supply the following
\begin{definition}[Local Minimiser]
An element $u\in\bd(\Omega)$ is called a \emph{local $\bd$--minimiser} of $\mathfrak{F}$ if and only if for every open Lipschitz subset $\omega$ of $\Omega$ there holds 
\begin{align}
\widetilde{\mathfrak{F}}_{u}[u;\omega]\leq \widetilde{\mathfrak{F}}_{u}[v;\omega]
\end{align}
for all $v\in\bd(\omega)$. 
\end{definition}
\begin{remark}[$\bd$--minima are local $\bd$--minima] 
Let $u\in\bd(\Omega)$ be a $\bd$--minimiser of $\mathfrak{F}$. For a given Lipschitz subset $\omega$ of $\Omega$ and an arbitrary $v\in\bd(\omega)$, we put $w:=\mathbbm{1}_{\omega}v+\mathbbm{1}_{\Omega\setminus\overline{\omega}}u$. By the classical $\bd$--extension principle, $w\in\bd(\Omega)$. By definition of $\bd$--minima, we have $\mathfrak{F}_{u_{0}}[u]\leq \mathfrak{F}_{u_{0}}[u+ v]$ and thus, working from the definition of $\mathfrak{F}_{u_{0}}$,  
\begin{align*}
\mathfrak{F}_{u_{0}}[u] & =\int_{\omega}f(\mathscr{E}u)\dif x + \int_{\Omega\setminus\overline{\omega}}f(\mathscr{E}u)\dif x + \int_{\omega}f^{\infty}\left(\frac{\dif \E u}{\dif |\E^{s}u|}\right)\dif |\E^{s}u|  \\ & + \int_{\Omega\setminus\overline{\omega}}f^{\infty}\left(\frac{\dif \E u}{\dif |\E^{s}u|}\right)\dif |\E^{s}u| + \int_{\partial\Omega}f^{\infty}(\nu_{\partial\Omega}\odot\trace(u-u_{0}))\dif\mathcal{H}^{n-1} \\ & \leq  \mathfrak{F}_{u_{0}}[w] = \int_{\omega}f(\mathscr{E}v)\dif x + \int_{\Omega\setminus\overline{\omega}}f(\mathscr{E}u)\dif x + \int_{\omega}f^{\infty}\left(\frac{\dif \E v}{\dif |\E^{s}v|}\right)\dif |\E^{s}v|  \\ & + \int_{\Omega\setminus\overline{\omega}}f^{\infty}\left(\frac{\dif \E u}{\dif |\E^{s}u|}\right)\dif |\E^{s}u| + \int_{\partial\Omega}f^{\infty}(\nu_{\partial\Omega}\odot\trace(u-u_{0}))\dif\mathcal{H}^{n-1}\\ & + \int_{\partial\omega}f^{\infty}(\trace(v-u)\odot \nu_{\partial\omega})\dif\mathcal{H}^{n-1}, 
\end{align*} 
where $a\odot b:=\frac{1}{2}(ab^{\mathsf{T}}+ba^{\mathsf{T}})$ for $a,b\in\R^{n}$ is the symmetrised tensor product. Now, the obvious cancellations yield that $u$ is a local $\bd$--minimiser. 
\end{remark}
Usually, within the setup of Sobolev-- or $\ld$--spaces, local minima are defined as elements which, for any ball $\ball\subset\Omega$, minimise $\mathfrak{F}[v;\ball]:=\int_{\Omega}f(\nabla v)\dif x$ with respect to its own boundary values: $\mathfrak{F}[v;\ball]\leq \mathfrak{F}[v+\varphi;\ball]$ for all $\varphi\in\sobo_{0}^{1,1}(\Omega;\R^{N})$. In the $\bd$--setup, however, we have to keep track of the possible non--zero singular part and hence the boundary penalisation term in the relaxed functional. 

As a matter of fact, because the variational problem \eqref{eq:varprin} is genuinely vectorial, we cannot expect generalised minima to be globally of class $\hold^{1,\alpha}$ (see \cite{Beck,Giusti,Mingione} for comprehensive lists of counterexamples to full regularity) unless $f$ is of special structure (so for instance of Uhlenbeck--type, that is, has radial structure: $f=g(|\cdot|^{2})$ for some $g\in\hold^{2}(\R)$). A natural and by now well--established substitute hence is the notion of \emph{partial regularity}. Here, we say that a mapping $v\colon\Omega\to\R^{n}$ is \emph{partially regular} if there exists an open subset $\Omega_{v}$ of $\Omega$ such that $\mathscr{L}^{n}(\Omega\setminus\Omega_{v})=0$ and $v$ is of class $\hold^{1,\alpha}$ in a neighbourhood of any element of $\Omega_{v}$ for any $0<\alpha<1$; consequently, $\Omega_{v}$ is called the \emph{regular set} and the relative complement $\Sigma_{v}$ the \emph{singular set}. To conclude with, the aim of the present paper is to show that under fairly weak ellipticity assumptions imposed on the variational integrand $f$, generalised minima or $\bd$--minima of variational integrals \eqref{eq:varprin} are partially regular, a result to whose precise description we turn now.
\subsection{Main Result \& Organisation of the Article}
We now pass to the description of our main result which, in turn, is a generalisation of Theorem 1.1 in Anzellotti \& Giaquinta's article \cite{AG}. The main feature of this result is that the convexity assumptions is rather weak and does not rely on any uniform bounds on the second derivatives from below. 
\begin{theorem}\label{thm:1}
Let $f\in \hold^{2}(\R_{\sym}^{n\times n};\R_{\geq 0})$ be convex and of linear growth. Suppose that $u\in\bd(\Omega)$ is a local $\bd$--minimiser of $\mathfrak{F}$ given by \eqref{eq:varprin}. If $(x,z)\in\Omega\times\R_{\sym}^{n\times n}$ is such that 
\begin{align*}
\lim_{R\searrow 0}\left[\dashint_{\ball(x,R)}|\mathscr{E}u-z|\dif x + \frac{|\E^{s}u|(\ball(x,R))}{\mathscr{L}^{n}(\ball(x,R))}\right]=0
\end{align*}
and $f''(z)$ is positive definite, then $u\in \hold^{1,\alpha}(U;\R^{n})$ for a suitable neighbourhood $U$ of $x$ for all $0<\alpha<1$. 
\end{theorem}
The theorem immediately implies partial regularity: Indeed, by the Lebesgue differentiation theorem for Radon measures \cite[Chpt. 1.6, Thm. 1]{EvGa}, $\mathscr{L}^{n}$--a.e. $x\in\Omega$ is a Lebesgue point for $\bm{\varepsilon}(v)$ and so the assumptions of the previous theorem are satisfied. In conclusion, if we put 
\begin{align*}
\widetilde{\Omega}_{v}:=\left\{x\in\Omega\colon\;\lim_{R\searrow 0}\dashint_{\ball(x,R)}|\mathscr{E}v-(\mathscr{E}v)_{\ball(x,R)}|\dif y +\frac{|\E^{s}v|(\ball(x,R))}{\mathscr{L}^{n}(\ball(x,R))}= 0 \right\}
\end{align*}
we consequently gain $\Omega_{v}=\widetilde{\Omega}_{v}$, openness of $\Omega_{v}$ together with $\mathscr{L}^{n}(\Sigma_{v})=0$. Once this is achieved, we use boundedness of singular integrals of convolution type on the H\"{o}lder--Zygmund classes $\besov_{\infty,\infty}^{s}$ to deduce that if $\sg(v)$ is of class $\hold^{0,\alpha}$ in a neighbourhood of some $x\in\Omega$, then so is $\D v$ in the same neighbourhood. Using the same method, we also obtain a regularity result for integrands of $p$--growth, $1<p<\infty$:
\begin{theorem}\label{thm:2}
Let $f\in \hold^{2}(\R_{\sym}^{n\times n};\R_{\geq 0})$ be convex and satisfy, for some $1<p<\infty$, $c_{1}|\xi|^{p}\leq f(\xi)\leq c_{2}(1+|\xi|^{p})$ for all $\xi\in\R_{\sym}^{n\times n}$ with two constants $0<c_{1}\leq c_{2}<\infty$. Suppose that $u\in\sobo^{1,p}(\Omega;\R^{n})$ is a local minimiser of $\mathfrak{F}$ given by \eqref{eq:varprin}. If $f''(z)$ is positive definite for any $z\in\R_{\sym}^{n\times n}$, then there exists an open subset $\Omega_{u}$ of $\Omega$ such that $\mathscr{L}^{n}(\Omega\setminus\Omega_{u})=0$ and $u$ is of class $\hold^{1,\alpha}$ in a neighbourhood of any of the elements of $\Omega_{u}$. 
\end{theorem}
The novelty of the previous result is that, albeit partial regularity results in the superlinear growth regime are well--known, they usually rely on strong convexity assumptions; see \cite{FS}. Here, however, we only require positive definiteness of $f''$. 

In proving the previous two theorems, we shall pursue an unified approach for both $p=1$ and $1<p<\infty$. To do so, we collect in section \ref{sec:prelims} preliminary definitions and results and consequently turn to specific types of Poincar\'{e}--inequalities in section \ref{sec:poincare} which shall be instrumental in the proof of Theorems \ref{thm:1} and \ref{thm:2}. This, in turn, is accomplished in section \ref{sec:pr}. 
\subsection{Method of Proof and Comparison to other Strategies}
The method we make use of is entirely direct and relies on comparing minima to its mollifications which, in turn, are shown to solve suitable linear elliptic systems, see section \ref{sec:pr}. It is not clear to us how to establish Theorem \ref{thm:1} by means of other methods of vectorial elliptic regularity theory as most of them crucially rely on compactness arguments; see \cite{Beck} for a comprehensive overview. These, however, seem hard to be employed in the present setup. So, for instance, assuming a strong convexity condition such as, e.g., that of $\mu$--ellipticity (see \cite{GK,Bild}), \textsc{De Giorgi}'s blow--up technique would yield partial regularity if we already knew that the minimiser is of class $\ld$ which, however, in the fairly general setup of Theorem \ref{thm:1} is not the case. For a detailled discussion of such issues, the interested reader is referred to \cite{Gm}. 
\begin{center}
\textbf{Acknowledgment}
\end{center}
The author warmly thanks his advisor \textsc{Jan Kristensen} for proposing the theme of the present paper and for numerous discussions throughout. Moreover, he is indebted to \textsc{Lars Diening} for various discussions regarding Poincar\'{e}--type inequalities. 
\section{Preliminaries}\label{sec:prelims}
This section fixes notation and collects facts about $\bd$--functions, convex functions of measures and various comparison estimates that will prove convenient in the partial regularity proof below.  
\subsection{General Notation}
Unless otherwise stated, we assume $\Omega$ to be an open and bounded subset of $\R^{n}$ with Lipschitz boundary $\partial\Omega$. Given $x_{0}\in\R^{n}$ and $R>0$, we denote $\ball(x_{0},R):=\ball_{R}(x_{0}):=\{x\in\R^{n}\colon |x-x_{0}|<r\}$ the open ball with radius $R$ centered at $x_{0}$. Given two finite dimensional vector spaces $V$ and $E$, the space $\hom(V;E)$ of linear mappings $T\colon V\to E$ is equipped with the Hilbert--Schmidt norm; then we denote for $R>0$ and $A\in\hom(V;E)$ the open ball with radius $R>0$ centered at $A$ by $\mathbb{B}_{R}(A)$. The $n$--dimensional Lebesgue measure is denoted $\mathscr{L}^{n}$ and the $(n-1)$--dimensional Hausdorff measure is denoted $\mathscr{H}^{n-1}$. Given two positive, real valued functions $f,g$, we indicate by $f\lesssim g$ that $f\leq C g$ with a constant $C>0$. If $f\lesssim g$ and $g\lesssim f$, we write $f\approx g$. Finally, given a measurable set $\mathcal{O}$ with $0<\mathscr{L}^{n}(\mathcal{O})<\infty$ and a measurable function $u\colon \mathcal{O}\to\R^{m}$ for $m\in\mathbb{N}$, we use the equivalent notations 
\begin{align*}
(u)_{\mathcal{O}}:=\dashint_{\mathcal{O}}u \dif x := \frac{1}{\mathscr{L}^{n}(\mathcal{O})}\int_{\mathcal{O}}u\dif x
\end{align*}
whenever the last expression is finite. If $\mathcal{O}=\ball(z,r)$ is a ball, we also use the shorthand $(u)_{z,r}:=(u)_{\ball(z,r)}$. Lastly, $\langle\cdot,\cdot\rangle$ denotes the euclidean inner product in $\R^{m}$ throughout, where the explicit value of $m$ will be clear from the context.
\subsection{Functions of Bounded Deformation}
In this subsection we gather various facts on functions of bounded deformation that shall prove important in the main part of the paper. Let $\Omega$ be an open subset of $\R^{n}$. We say that a measurable mapping $v\colon\Omega\to\R^{n}$ belongs to $\bd(\Omega)$ and is then said to be of \emph{bounded deformation} if and only if $v\in\lebe^{1}(\Omega;\R^{n})$ and its symmetric distributional gradient can be represented by a finite $\R_{\sym}^{n\times n}$--valued Radon measure. By the Riesz representation theorem for Radon measures, it is sraightforward to verify that the latter amounts to requiring 
\begin{align*}
|\sg(v)|(\Omega):=\sup\left\{\int_{\Omega}\langle u,\di(\varphi)\rangle\dif x\colon\;\varphi\in\hold_{c}^{1}(\Omega;\R_{\sym}^{n\times n},\;|\varphi|\leq 1 \right\}<\infty. 
\end{align*}
The usual norm on $\bd(\Omega)$ is given by $\|v\|_{\bd(\Omega)}:=\|v\|_{\lebe^{1}(\Omega;\R^{n})}+|\sg(v)|(\Omega)$, however, the norm topology is too strong for most applications. Instead, we consider two other notions of convergence defined as follows. We say that a sequence $(v_{k})\subset\bd(\Omega)$ converges to $v\in\bd(\Omega)$ in the \emph{weak*--sense} (in symbols $v_{k}\wstar v$) if and only if $v_{k}\to v$ strongly in $\lebe^{1}(\Omega;\R^{n})$ and $\sg(v_{k})\wstar \sg(v)$ as $k\to\infty$ in the sense of $\R^{n\times n}$--valued Radon measures on $\Omega$. Moreover, we say that $(v_{k})$ converges to $v$ \emph{strictly} provided $v_{k}\wstar v$ and $|\sg(v_{k})|(\Omega)\to|\sg(v)|(\Omega)$ as $k\to\infty$. 

Let us further remark that if $\Omega$ is Lipschitz, then sequences which are bounded in $\bd(\Omega)$ allow to extract subsequences which converge to some $v\in\bd(\Omega)$ in the weak*--sense. This is a simple consequence of the Rellich--Kondrachov Theorem in conjunction with the Banach--Alaoglu Theorem. 
It is important to note that $\bd(\Omega)\cap\hold^{\infty}(\Omega;\R^{n})$ is dense with respect to strict convergence. 

Due to \textsc{Strang \& Temam} \cite{ST} (also see \cite{AG1}), elements of $\bd(\Omega)$ attain traces in the $\lebe^{1}$--sense at the boundary provided $\Omega$ is an open and bounded Lipschitz subset of $\R^{n}$. Denoting the corresponding boundary trace operator $\trace\colon \bd(\Omega)\to \lebe^{1}(\partial\Omega;\R^{n})$, we record that $\trace_{\bv}=\trace|_{\bv}$ with the usual trace operator $\trace_{\bv}$ on $\bv(\Omega;\R^{n})$. Hence, by a result of \textsc{Gagliardo}, the trace operator $\trace\colon \bd(\Omega)\to\lebe^{1}(\partial\Omega;\R^{n})$ is onto indeed. Moreover, the Gauss--Green formula takes the following form: Given $\varphi=(\varphi_{lm})_{l,m=1,...,n}\in\hold_{c}^{\infty}(\R^{n};\R_{\sym}^{n\times n})$ and $u\in\bd(\Omega)$, we have 
\begin{align}\label{eq:trace}
\sum_{l,m=1}^{n}\int_{\Omega}\varphi_{lm}\dif \E_{lm}u+\int_{\Omega}\langle u,\di(\varphi)\rangle\dif x = \int_{\partial\Omega}\langle \trace(u)\odot\nu_{\partial\Omega},\varphi\rangle\dif \mathcal{H}^{n-1}, 
\end{align}
where $\nu_{\partial\Omega}$ is the outer unit normal to $\partial\Omega$. Lastly, it is important to note that the trace operator is continuous for strict convergence indeed whereas this is not so for weak*--convergence as specified above; in fact, we have $\overline{\hold_{c}^{\infty}(\Omega;\R^{n})}^{\text{w}^{*}}=\bd(\Omega)$. Regarding the fine properties of $\bd$--functions, we refer the reader to \cite{ACD}. We conclude with following variant of Poincar\'{e}'s Inequality in $\bd$ which will serve as an important auxiliary tool in the following. We denote for an open subset $U$ of $\R^{n}$ the \emph{rigid deformations}
\begin{align}
\mathcal{R}(U):=\big\{u\colon U\to\R^{n}\colon\; \sg(u)\equiv 0\;\text{in}\;U\big\}
\end{align}
and note that if $U$ is connected, then any element $u\in\mathcal{R}(U)$ is of the form $u(x)=Ax+b$ with scew--symmetric $A\in\R^{n\times n}$ and $b\in\R^{n}$. 
\begin{lemma}\label{lem:poincare}
Let $Q\neq\{0\}$ be an open cube in $\R^{n}$. There exists a constant $c=c(n)>0$ such that for all $v\in\bd(\Omega)$ there exists $r\in \mathcal{R}(Q)$ such that 
\begin{align*}
\|v-r\|_{\lebe^{1}(Q;\R^{n})}=\inf_{\widetilde{r}\in\mathcal{R}(U)}\|v-\widetilde{r}\|_{\lebe^{1}(Q;\R^{n})}\leq c\ell(Q)|\sg(v)|(\overline{Q}),  
\end{align*}
where $\ell(Q)$ is the sidelength of the cube $Q$. 
\end{lemma}
\subsection{Convex Functions of Measures}\label{sec:cfom}
In this section we recall the rudiments of applying functions to measures as it is required for the subsequent paragraphs, and refer the reader to \textsc{Demengel} \& \textsc{Temam} \cite{DT} and \textsc{Anzellotti} \cite{Anz} for more detail. Let $f\colon\R^{m}\to\R_{\geq 0}$ be a continuous and convex functions satisfying \eqref{eq:lingrowth}. Under these assumptions, the recession function given by \eqref{eq:recession} exists (with the obvious modification) and is a positively $1$--homogeneous and convex function. Now let $\Omega\subset\R^{n}$ be open and let $\mu$ be an $\R^{m}$--valued finite Radon measure on $\Omega$ with $(\mathscr{L}^{n},\mu)\ll\nu$. Let $\nu$ be a non--negative, finite Radon measure on $\Omega$ and consider its Radon--Nikod\v{y}m decomposition $\mu=\mu_{\nu}^{ac}+\mu_{\nu}^{s}=\frac{\dif\mu}{\dif\nu}\nu+\frac{\dif\mu}{\dif |\mu^{s}|}|\mu^{s}|$ with respect to $\nu$ , where $\mu_{\nu}^{ac}$ and $\mu_{\nu}^{s}$ are the absolutely continuous and singular parts of $\mu$ with respect to $\nu$, respectively. We define a new measure $f(\mu)$ on the Borel $\sigma$--algebra $\mathscr{B}(\Omega)$ by 
\begin{align}\label{eq:functionofmeasure}
f(\mu)(A):= \int_{A}f^{\#}\left(\frac{\dif\mathscr{L}^{n}}{\dif\nu},\frac{\dif\mu}{\dif\nu}\right)\dif\nu,
\end{align}
where $|\mu^{s}|$ is the total variation measure associated with the singular part of $\mu$ with respect to $\nu$, where $f^{\#}\colon [0,\infty)\times\R^{m}\to\R$ is the \emph{homogenised} or \emph{perspective integrand} which, for $t\geq 0$ and $\xi\in\R^{m}$ is given by 
\begin{align*}
f^{\#}(t,\xi):=\begin{cases} tf(\xi/t)&\;\text{given}\,t>0,\\
f^{\infty}(\xi)&\;\text{given}\;t=0. 
\end{cases}
\end{align*}
Since $f^{\#}$ is positively $1$--homogeneous, \eqref{eq:functionofmeasure} is well--defined indeed, in particular, it does not depend on the choice of $\nu$. It is easy to see that if we put $\nu:=|\mu^{s}|+\mathscr{L}^{n}$, where $\mu^{s}$ is the singular part of $\mu$ with respect to $\mathscr{L}^{n}$, then \eqref{eq:functionofmeasure} becomes 
\begin{align*}
f(\mu)(A)=\int_{A}f\left(\frac{\dif\mu}{\dif\mathscr{L}^{n}}\right)\dif\mathscr{L}^{n}+\int_{A}f^{\infty}\left(\frac{\dif \mu}{\dif |\mu^{s}|}\right)\dif |\mu^{s}|\qquad\text{for}\;A\in\mathscr{B}(\Omega). 
\end{align*}
The following theorem of \textsc{Reshetnyak} \cite{Reshetnyak}, which is instrumental for all of what follows, gives necessary conditions for the functionals from above to be lower semicontinuous with respect to weak*-- and area--strict topologies:
\begin{theorem}[Reshetnyak Continuity Theorem]\label{thm:resh}
Suppose $F\colon\R^{m}\to\R_{\geq 0}$ is strictly convex, positively $1$--homogeneous and is of linear growth, i.e., satisfies \eqref{eq:lingrowth} with the obvious modifications. If $(\mu_{j})$ is a sequence of $\R^{m}$--valued finite Radon measures on $\Omega$ such that both $\mu_{j}\wstar\mu$ as $j\to\infty$ and $F(\mu_{j})(\Omega)\to F(\mu)(\Omega)$ as $j\to\infty$ for some Radon measure $\mu$ on $\Omega$, then there holds $G(\mu_{j})(\Omega)\to G(\mu)(\Omega)$ for any positively homogeneous function $G\colon\R^{m}\to\R_{\geq 0}$ satisfying \eqref{eq:lingrowth} (with the obvious modification). 
\end{theorem}
As a matter of fact, convexity of the area intgrand $\m_{2}(\xi):=\sqrt{1+|\xi|^{2}}$ allows to introduce the notion of area--strict convergence: We say that a sequence $(v_{k})\subset\bd(\Omega)$ converges to $v\in\bd(\Omega)$ in the \emph{area--strict sense} (in symbols $v_{j}\stackrel{\langle\cdot\rangle}{\to}v$) provided $v_{k}\to v$ strongly in $\lebe^{1}(\Omega;\R^{n})$ and $\sqrt{1+|\sg(v_{k})|^{2}}(\Omega)\to\sqrt{1+|\sg(v)|^{2}}(\Omega)$ as $k\to\infty$. As an immediate consequence of Theorem \ref{thm:resh}, we obtain 
\begin{lemma}\label{lem:area}
Let $(v_{j})\subset\bd(\Omega)$ and $v\in\bd(\Omega)$ be such that $v_{j}\to v$ in the area--strict sense as $j\to\infty$. Then for any convex function $F\colon\R_{\sym}^{n\times n}\to\R_{\geq 0}$ that satisfies \eqref{eq:lingrowth} there holds $F[\sg(v_{j})](\Omega)\to F[\sg(v)](\Omega)$ as $j\to\infty$. 
\end{lemma}
Finally, we shall make use of the following area--strict approximation result whose proof appears as a straightforward modification of the arguments given, e.g., in \cite[Lemma A.3.1]{FS}:
\begin{lemma}[Area--Strict Smooth Approximation]
Let $\Omega$ be an open and bounded Lipschitz subsets of $\R^{n}$. Given $u\in\bd(\Omega)$, there exists $(u_{j})\subset (\hold^{\infty}\cap\ld)(\Omega)$ such that $v_{j}\stackrel{\langle\cdot\rangle}{\to}v$ as $j\to\infty$. 
\end{lemma}
\subsection{Decay Estimates for Linear Systems}
Aiming for comparison with solutions of linear elliptic systems that depend on the symmetric gradients in the partial regularity proof, we collect now decay estimates for such mappings. To this end, let $\ball\subset\R^{n}$ be a ball and consider for $w\in\sobo^{1,2}(\ball;\R^{n})$ the variational principle 
\begin{align}\label{eq:linearproblem}
\text{to minimise}\;\;\mathfrak{F}[v]:=\int_{\Omega}g(\sg(v))\dif x\qquad \text{over}\;v\in w+\sobo_{0}^{1,2}(\Omega;\R^{n}), 
\end{align}
where $g(\xi):=\langle\mathcal{A}\xi,\xi\rangle + \langle b,\xi\rangle + c$ is a polynomial of degree two with arbitrary $b\in\R^{n}$ and $c\in\R$, whereas we assume that $\mathcal{A}\colon\R^{n\times n}\times\R^{n\times n}\to\R$ is elliptic in that there exists $\ell>0$ such that $\langle\mathcal{A}\xi,\xi\rangle \geq \ell |\xi|^{2}$ for all $\xi\in\R^{n\times n}$. As a small variation of \cite[Lemma 3.0.5]{FS}, we provide the following lemma for the reader's convenience:
\begin{lemma}\label{lem:lineardecayestimates}
There exists a unique solution $u\in w+\sobo_{0}^{1,2}(\ball;\R^{n})$ of \eqref{eq:linearproblem}. Moreover, this solution satisfies the following: 
\begin{enumerate}
\item There exists a constant $c=c(n,|\mathcal{A}|,\ell)>0$ such that if $\ball(z,R)\subset\Omega$, then for all $0<r<R$ there holds 
\begin{align*}
\int_{\ball(z,r)}|\adop (u)-(\adop (u))_{z,r}|^{2}\dif x \leq c\left(\frac{r}{R}\right)^{n+2}\int_{\ball(z,R/2)}|\adop (u)-(\adop (u))_{z,R/2}|^{2}\dif x. 
\end{align*}
\item There exists a constant $c=c(n,|\mathcal{A}|,\ell)>0$ such that if $\ball(z,R)\subset\Omega$ and $w\in\hold^{1,\alpha}(\Omega;\R^{n})$, then 
\begin{align*}
[\sg(u)]_{\hold^{0,\alpha}(\ball(z,R/2);\R^{n\times n})}\leq c [\sg(w)]_{\hold^{0,\alpha}(\ball(z,R/2);\R^{n\times n})}.
\end{align*}
\end{enumerate}
\end{lemma}
\begin{proof}
Let $\mathcal{A}\colon\R^{(n\times n)\times (n\times n)}\to\R$ be an elliptic bilinear form. We consider the linear system 
\begin{align}\label{eq:linearisedsystem}
-\di(\mathcal{A}(\sg(u)))=\di(f)
\end{align}
for $f\in\lebe^{2}(\Omega;\R^{n\times n})$. Let $u\in\sobo^{1,2}(\Omega;\R^{N})$ be a weak solution to this equation. \\

(i) Firstly assume that $f\equiv 0$. Let $\ball\Subset\Omega$ be an arbtirary ball and denotes $\frac{1}{2}\ball$ the ball with the same radius but half the radius. We then pick $\rho\in\hold_{c}^{\infty}(\ball;[0,1])$ with $\mathbbm{1}_{\frac{1}{2}\ball}\leq\rho\leq\mathbbm{1}_{\ball}$ and $|\nabla\rho|\leq 4/|\ball|^{\frac{1}{n}}$. Let $s\in\{1,...,n\}$ and define for $|h|<\dista(\ball;\partial\Omega)$ the test function $\varphi:=\Delta_{s,-h}(\rho^{2}\Delta_{s,h}u)$. Testing \eqref{eq:linearisedsystem} with $\varphi$, we obtain after integration by parts for difference quotients, Young's inequality and using ellipticity of $\mathcal{A}$ that
\begin{align*}
\int_{\frac{1}{2}\ball}|\Delta_{s,h}\sg(u)|^{2}\dif x \leq c(\rho)\int_{\ball}|\Delta_{s,h}u|^{2}\dif x.
\end{align*}
By Korn's inequality and the usual difference quotient--type characterisation of $\sobo^{1,2}$, the right side is uniformly bounded in $|h|$ and so is the left side. Passing $|h|\downarrow 0$ yields validity of the previous inequality with $\partial_{s}$ instead of $\Delta_{s,h}$. Note that $\partial_{s}\circ\sg$ is an elliptic first order differential operator and hence possesses a Green's function. In consequence, we obtain $u\in\sobo_{\locc}^{2,2}(\Omega;\R^{n})$ and, noticing that with $u$ also $\partial^{\alpha}u$ is a weak solution of \eqref{eq:linearisedsystem} (recall $f\equiv 0$), we obtain $u\in\sobo_{\locc}^{k,2}(\Omega;\R^{n})$ for any $k\in\mathbb{N}$. \\

(ii) Choosing the same cut--off function $\rho$ as in (i) and putting $\varphi:=\rho^{2}(u-\mathbf{r})$, where $\mathbf{r}\in\mathcal{R}(\ball)$ is an arbitrary rigid deformation, testing \eqref{eq:linearisedsystem} with $\varphi$ and using Young's inequality yields the Caccioppoli--type inequality 
\begin{align}\label{eq:Cacclin}
\int_{\frac{1}{2}\ball}|\sg(u)|^{2}\dif x \leq \frac{C}{R^{2}}\int_{\ball}|u-\mathbf{r}|^{2}\dif x, 
\end{align}
where $C=C(n,|\mathbb{A}|,\ell)>0$ is a constant. Moreover, choosing $\varphi:=\rho^{2}(u-(u)_{R})$, we obtain by use of Korn's inequality in the zero--boundary value version
\begin{align*}
\int_{\Omega}|\rho\D u|^{2}&+2\langle 2\rho\nabla\rho\odot u,\rho\D u\rangle\dif x + |2\rho\nabla\rho\odot u|^{2}\dif x = \int_{\Omega}|\D\,(\rho^{2}(u-(u)_{R})|^{2}\dif x  \\ & \leq \int_{\Omega}|\sg(\rho^{2}(u-(u)_{R}))|^{2}\dif x, 
\end{align*}
from which we deduce 
\begin{align}\label{eq:caccfull}
\int_{\frac{1}{2}\ball}|\D u|^{2}\dif x \leq \frac{C}{|\ball|^{\frac{2}{n}}}\int_{\ball}|u-(u)_{\ball}|^{2}\dif x. 
\end{align}
(iii) By Morrey's embedding, $\|u\|_{\lebe^{\infty}(\ball_{R/2};\R^{n})}\leq c \|u\|_{\sobo^{n,2}(\ball_{R/2};\R^{n})}$. In consequence, iterating steps (i) and (ii) with $\mathbf{r}\equiv 0$, we obtain for $r<R/2$
\begin{align*}
\int_{\ball_{r}}|u|^{2}\dif x \leq cr^{n}\|u\|_{\lebe^{\infty}(\ball_{R/2};\R^{n})}^{2} \leq cr^{n}\|u\|_{\sobo^{n,2}(\ball_{R/2};\R^{n})}^{2} \leq c(R)r^{n}\|u\|_{\lebe^{2}(\ball_{R};\R^{n})}^{2},  
\end{align*}
and a simple rescaling argument yields that $C(R)=C/R^{n}$. Now note that if $u$ is a weak solution, then so is $u-\mathbf{r}$, where $\mathbf{r}\in\mathcal{R}(\Omega)$ is the uniquely determined rigid deformation such that $\int_{\ball}|u-\mathbf{r}|^{2}\dif x \leq c \int_{\ball}|\sg(u)|^{2}\dif x$ (as $\mathbf{r}$ is an affine function, its extension from $\ball$ to $\Omega$ is trivial). In consequence, we obtain by use of the Sobolev--Poincar\'{e} Inequality in its symmetric gradient form
\begin{align*}
\int_{\ball_{r}}|u-\mathbf{r}|^{2}\dif x & \leq r^{2}\int_{\ball_{r}}|\sg(u)|^{2}\dif x \leq c r^{2}\Big(\frac{r}{R}\Big)^{n}\int_{\ball_{R}}|u-\mathbf{r}|^{2}\dif x \leq c \Big(\frac{r}{R}\Big)^{n+2}\int_{\ball_{R}}|u-\mathbf{r}|^{2}\dif x. 
\end{align*}
Moreover, invoking the usual Poincar\'{e} inequality in conjunction with \eqref{eq:caccfull}, we obtain similarly 
\begin{align*}
\int_{\ball_{r}}|u-(u)_{r}|^{2}\dif x \leq c \left(\frac{r}{R}\right)^{n+2}\int_{\ball_{R}}|u-(u)_{R}|^{2}\dif x
\end{align*}
and record that, since with $u$ also $\partial_{s}u$ for any $s\in\{1,...n\}$ is a weak solution, the previous estimate holds true when consequently replacing $u$ by $\partial_{s}u$. Put $w(x):=u(x)-(\sg(u))_{R}x$ and let $\widetilde{\mathbf{r}}\in\mathcal{R}(\Omega)$ be such that $\|w-\widetilde{\mathbf{r}}\|_{\lebe^{2}(\ball_{R};\R^{n})}^{2}\leq CR^{2}\|\sg(w)\|_{\lebe^{2}(\ball_{R};\R^{n\times n})}^{2}$ and put $v:=w-\widetilde{\mathbf{r}}$, $V:=(\sg(u))_{R}+\D\widetilde{\mathbf{r}}$. Then we have for $r\leq \frac{R}{2}$
\begin{align*}
\int_{\ball_{r}}|\sg(u)-(\sg(u))_{r}|^{2}\dif x & \stackrel{a(x):=(\D u)_{r}x}{\leq} \int_{\ball_{r}}|\D\,(u-a)|^{2}\dif x = \int_{\ball_{r}}|\D u -(\D u)_{r}|^{2}\dif x \\
& \leq C \left(\frac{r}{R}\right)^{n+2}\int_{\ball_{R/2}}|\D u - (\D u)_{R/2}|^{2}\dif x \\
&\leq C \left(\frac{r}{R}\right)^{n+2}\int_{\ball_{R/2}}|\D v|^{2}\dif x \leq C\left(\frac{r}{R}\right)^{n+2}\frac{1}{R^{2}}\int_{\ball_{R}}|v|^{2}\dif x  \\
& \leq C\left(\frac{r}{R}\right)^{n+2}\int_{\ball_{R}}|\sg(w)|^{2}\dif x \\ & = C\left(\frac{r}{R}\right)^{n+2}\int_{\ball_{R}}|\sg(u)-(\sg(u))_{R}|^{2}\dif x.
\end{align*}

(iv) We now pass to the case $f\not\equiv 0$. We write $u=v+w$, where $v\in u + \sobo_{0}^{1,2}(\ball_{R};\R^{N})$ and $v$ solves $\di(\mathcal{A}(\sg(v)))=0$ in $\ball_{R}$ and $w$ solves $\di(\mathcal{A}(\sg(w)))=\di(f)$ weakly in $\ball_{R}$. We consequently obtain (as $(\sg(u))_{\ball_{R}}=\argmin_{a\in\R^{n\times n}}\int_{\ball_{R}}|\sg(u)-a|^{2}\dif x$)
\begin{align*}
\int_{\ball_{r}}|\sg(u)-(\sg(u))_{\ball_{r}}|^{2}\dif x & \leq \int_{\ball_{r}}|\sg(u)-(\sg(v))_{\ball_{r}}|^{2}\dif x \\ & \leq 2 \int_{\ball_{r}}|\sg(v)-(\sg(v))_{\ball_{r}}|^{2}\dif x + 2\int_{\ball_{r}}|\sg(w)|^{2}\dif x \\
& \leq c\Big(\frac{r}{R}\Big)^{n+2}\int_{\ball_{R}}|\sg(u)-(\sg(u))_{\ball_{R}}|^{2}\dif x + c\int_{\ball_{r}}|\sg(w)|^{2}\dif x \\
& \leq c\Big(\frac{r}{R}\Big)^{n+2}\int_{\ball_{R}}|\sg(u)-(\sg(u))_{\ball_{R}}|^{2}\dif x + c\int_{\ball_{R}}|f-(f)_{\ball_{R}}|^{2}\dif x. 
\end{align*} 
In particular, if $f$ is constant, then the last term in the previous inequality vanishes identically and the claim follows. \\

(v) We complete the proof by showing (b) and firstly note the weak solutions of the systems
\begin{align*}
(\text{S}_{1})\begin{cases}
-\di(\mathcal{A}(\sg(u)))=0&\;\text{in}\;\Omega,\\
u=w&\;\text{on}\;\partial\Omega
\end{cases},\quad (\text{S}_{2})\begin{cases}
-\di(\mathcal{A}(\sg(\widetilde{u})))=-\di(\mathcal{A}(\sg(w)))&\;\text{in}\;\Omega,\\
\widetilde{u}=0&\;\text{on}\;\partial\Omega
\end{cases},
\end{align*}
are linked by $u=\widetilde{u}-w$. Put $f:=\mathcal{A}(\sg(w))$ in (iv) to deduce for a fixed $0<\tau<1$ with $\Phi(s):=\int_{\ball_{s}}|\sg(u)-(\sg(u))_{r}|^{2}\dif x$
\begin{align*}
\Phi(\tau R)& \leq c\tau^{n+2}\Phi(R)+ cR^{n+2\alpha}[\mathcal{A}(\sg(w))]_{\mathcal{L}^{2,n+2\alpha}(\Omega;\R^{N\times n})}. 
\end{align*}
In this situation, the iteration lemma \cite[Lemma B.3]{Beck} completes the proof by use of the Campanato characterisation $\mathcal{L}^{p,n+\alpha p}\cong\hold^{0,\alpha}$ of H\"{o}lder continuity. 
\end{proof}
\section{Poincar\'{e}--Inequalities}\label{sec:poincare}
\begin{figure}
\centering
\begin{tikzpicture}
\draw [<->] (3.6, -4.0) -- (4.5,-4.0);
\node at (4.1,-4.4) {$\frac{\varepsilon}{\sqrt{n}}$};
\draw [<->] (4.85, -2.7) -- (4.85,-3.6);
\node at (5.07,-3.25) {$\frac{\varepsilon}{\sqrt{n}}$};
\node at (-3.8,-4.0) {$\Gamma_{\varepsilon}:=z+\frac{\varepsilon}{\sqrt{n}}\mathbb{Z}^{n}$};
\node at (0.6,-0.3) {$\Omega$};
\draw [<->] (-1.4, -1.71) -- (-1.885,-3.4);
\node [color=black] at (-1.9,-2.5) {$2\varepsilon$};
\node [ColorPink] at (-1,1.1) {\textbullet};
\node [ColorPink] at (-0.5,1.25) {\textbullet};
\node [ColorPink] at (-0.0,1.30) {\textbullet};
\node [ColorPink] at (0.5,1.25) {\textbullet};
\node [ColorPink] at (1,1.1) {\textbullet};
\node [ColorPink] at (1.4,0.9) {\textbullet};
\node [ColorPink] at (-1.2,-2.4) {\textbullet};
\node [ColorPink] at (-0.93,-2.45) {\textbullet};
\node [ColorPink] at (-0.7,-2.4) {\textbullet};
\node [ColorPink] at (1.4,-2.4) {\textbullet};
\node [ColorPink] at (1.7,-2.33) {\textbullet};
\node [ColorPink] at (2.0,-2.26) {\textbullet};
\node [ColorPink] at (2.3,-2.19) {\textbullet};
\node [ColorPink] at (2.6,-2.05) {\textbullet};
\node [ColorPink] at (2.8,-1.92) {\textbullet};
\draw[step=.9cm,gray,dotted] (-5.0,-4.0) grid (5.0,3.0);
\shadedraw[color=blue, opacity=0.3] plot[smooth cycle,thick] coordinates {
    (0:2.0)
    (25:1.9)
    (50:2.0)
    (75:2.1)
    (90:2.0)
    (100:1.9)
    (110:1.7)
    (120:1.6) 
    (130:1.7)
    (140:2.0)  
    (150:2.2)
    (160:2.4)
    (170:2.5)
    (180:3.0)
    (190:3.2)
    (200:3.7)
    (210:3.5)
    (220:3.5)
    (240:3.0)
    (245:2.9)
    (250:2.8)
    (260:2.7) 
    (270:2.5)
    (275:2.4)
    (280:2.0)
    (290:2.4)
    (300:2.5)
    (305:3.1)
    (310:3.2)
    (320:3.0)    
    (330:3.3)
    (340:3.8)
    (350:4.0)    
  };
\node [color=black] at (-1.9,-2.5) {$2\varepsilon$};
\node [black] at (2.2,1.16) {$S_{\varepsilon}$};
\draw[color=black, opacity=0.5] plot[smooth cycle,thick] coordinates {
    (0:2.0)
    (25:1.9)
    (50:2.0)
    (75:2.1)
    (90:2.0)
    (100:1.9)
    (110:1.7)
    (120:1.6) 
    (130:1.7)
    (140:2.0)  
    (150:2.2)
    (160:2.4)
    (170:2.5)
    (180:3.0)
    (190:3.2)
    (200:3.7)
    (210:3.5)
    (220:3.5)
    (240:3.0)
    (245:2.9)
    (250:2.8)
    (260:2.7) 
    (270:2.5)
    (275:2.4)
    (280:2.0)
    (290:2.4)
    (300:2.5)
    (305:3.1)
    (310:3.2)
    (320:3.0)    
    (330:3.3)
    (340:3.8)
    (350:4.0)    
  };
\draw[line width=50pt, color=ColorPink, opacity=0.6] plot [smooth, tension=0.2] coordinates{
    (135:1.8)
    (140:2.0)
    (150:2.2)
    (160:2.4)
    (170:2.5)
    (180:3.0)
    (190:3.2)
    (200:3.7)
    (210:3.5)
    (220:3.5)
    (240:3.0)};
\draw[line width=50pt, color=ColorPink, opacity=0.6] plot [smooth, tension=0.2] coordinates{
    (330:3.3)
    (340:3.8)
    (350:4.0)
    (0:2.0)
    (25:1.9)};
\draw[line width=50pt, color=ColorPink, opacity=0.6] plot [smooth, tension=0.2] coordinates{
   (260:2.6)
   (265:2.5)
   (270:2.5)
   (275:2.4)
   (280:2.0)
   (290:2.4)
   (295:2.45)};
\end{tikzpicture}
\caption{Construction in the proof of Proposition \ref{prop:convest}}
\end{figure}
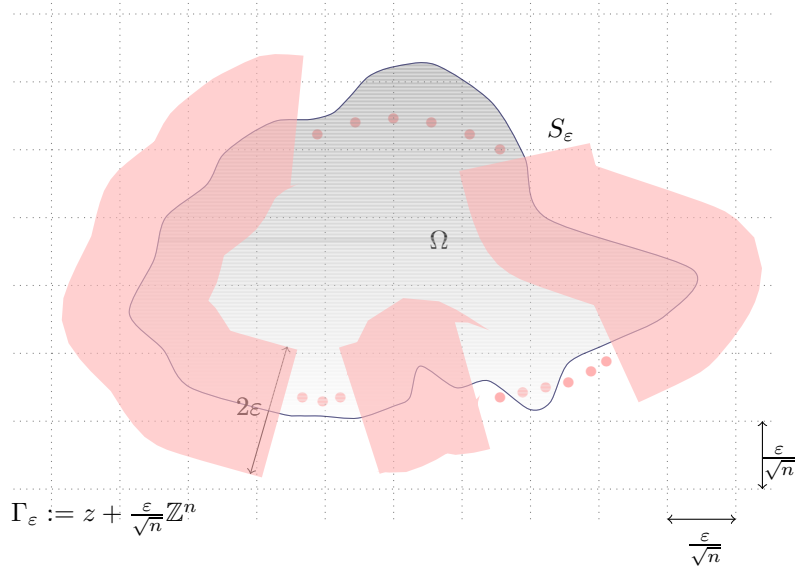
As usual, by a \emph{standard mollifier} we understand a positive, radially symmetric and continuous function $\rho\colon\ball(0,1)\to\R_{\geq 0}$ with $\|\rho\|_{\lebe^{1}(\ball(0,1))}=1$. 
\begin{proposition}\label{prop:convest}
Let $\widetilde{\Omega}\subset\R^{n}$ be open and suppose that $\Omega\Subset\tilde{\Omega}$ is a relatively compact subset. Let $\rho\in\hold_{c}^{\infty}(\ball(0,1);[0,1])$ be a standard mollifier. Then there exists a constant $C>0$ which only depends $\rho$ such that for all $0<\varepsilon<\dist(\Omega,\partial\widetilde{\Omega})$ and all $u\in\bd(\widetilde{\Omega})$ we have 
\begin{align}\label{eq:poincare}
\int_{\Omega}|u-\rho_{\varepsilon}*u|\dif x \leq C\varepsilon\,|\sg(u)|(\overline{N_{\varepsilon}(\Omega)}), 
\end{align}
where $N_{\varepsilon}(\Omega):=\{x\in\tilde{\Omega}\colon\;\dist(x,\partial\Omega)<\varepsilon\}$ is the $\varepsilon$--neighbourhood of $\Omega$. 
\end{proposition}
\begin{proof}
We first assume that $u\in(\hold^{1}\cap\ld)(\tilde{\Omega})$ and shall finally argue by strict denseness to obtain estimate \eqref{eq:poincare} for $\bd$--functions. Given a lattice $\Gamma_{z,\varepsilon}:=z+\frac{\varepsilon}{\sqrt{n}}\mathbb{Z}^{n}$ for $z\in\R^{n}$ and $\varepsilon>0$, we denote $\mathcal{Q}$ the collection of all cubes $Q$ which can be written as
\begin{align*}
Q=[\frac{\varepsilon}{\sqrt{n}} k_{1},\frac{\varepsilon}{\sqrt{n}}(k_{1}+1)]\times ... \times [\frac{\varepsilon}{\sqrt{n}} k_{n},\frac{\varepsilon}{\sqrt{n}}(k_{n}+1)],\qquad (k_{1},...,k_{n})\in\mathbb{Z}^{n}
\end{align*}
and have non--empty intersection with $\Omega$. Let us now note that if $\spt(v)\subset Q(x,r)$, then $\spt(\eta_{\varepsilon}*v)\subset Q(x,r+\varepsilon)$. We estimate
\begin{align*}
\int_{\Omega}|u-\rho_{\varepsilon}*u|\dif x & = \int_{\Omega}\left\vert \sum_{Q}\mathbbm{1}_{Q}(u-\rho_{\varepsilon}*u)\right\vert\dif x \\
& =  \int_{\Omega}\left\vert \sum_{Q}\mathbbm{1}_{Q}(u-R_{Q})+\mathbbm{1}_{Q}(\rho_{\varepsilon}*(R_{Q}-u))\right\vert\dif x \\
& \leq \sum_{Q}\int_{Q}|u-R_{Q}| \dif x  + \int_{\Omega}\left\vert \sum_{Q}\mathbbm{1}_{Q}(\rho_{\varepsilon}*(R_{Q}-u))\right\vert\dif x\\
& \leq C\ell(Q)\sum_{Q}\int_{Q}|\sg(u)| \dif x  + \int_{\R^{n}}\left\vert \sum_{Q}\mathbbm{1}_{Q}(\rho_{\varepsilon}*(R_{Q}-u))\right\vert\dif x \\
& \leq C\varepsilon\sum_{Q}\int_{Q}|\sg(u)| \dif x  + \sum_{Q}\int_{\R^{n}}|\mathbbm{1}_{Q}(\rho_{\varepsilon}*(R_{Q}-u))|\dif x \\
& \leq C\varepsilon\int_{\overline{\mathcal{U}_{\varepsilon}}}|\sg(u)| \dif x  + \sum_{Q}\int_{\R^{n}}|\mathbbm{1}_{Q}(\rho_{\varepsilon}*(R_{Q}-u))|\dif x, 
\end{align*}
where here in all of what follows all sums over cubes run over $Q\in\mathcal{Q}$ only. Now, we have due to $\int_{\R^{n}}\rho\dif x=1$
\begin{align*}
\sum_{Q}\int_{\R^{n}}|\mathbbm{1}_{Q}(\rho_{\varepsilon}*(R_{Q}-u))|\dif x & \leq \sum_{Q}\int_{\R^{n}}\int_{\R^{n}}\mathbbm{1}_{Q}(x)\rho_{\varepsilon}(x-y)|R_{Q}(y)-u(y)|\dif y\dif x\\
& \leq \sum_{Q}\int_{\R^{n}}\int_{\R^{n}}\mathbbm{1}_{Q}(y+z)\rho_{\varepsilon}(z)|R_{Q}(y)-u(y)|\dif y\dif z\\
& \leq \sum_{Q}\int_{Q+\ball(0,\varepsilon)}|R_{Q}(y)-u(y)|\dif y\\
& \leq \sum_{Q}\int_{Q+Q}|R_{Q}(y)-u(y)|\dif y. 
\end{align*}
On the other hand, by Poincar\'{e}'s inequality, Lemma \ref{lem:poincare},
\begin{align*}
\int_{Q+Q}|R_{Q}(y)-u(y)|\dif y & = b^{n}\int_{Q}|(R_{Q}-u)(b\xi)|\dif\xi\\
& \leq C b^{n+1}\ell(Q)\int_{Q}|\sg(u)(b\xi)|\dif\xi \leq C \ell(Q)\int_{Q+Q}|\sg(u)(y)|\dif y, 
\end{align*}
where $C>0$ only depends on $n$. In consequence, we obtain by $\ell(Q)=\varepsilon$
\begin{align*}
\sum_{Q}\int_{\R^{n}}|\mathbbm{1}_{Q}(\rho_{\varepsilon}*(R_{Q}-u))|\dif x \leq C\sum_{Q}\varepsilon\int_{Q+Q}|\sg(u)(y)|\dif y.
\end{align*}
We now claim that there exists a constant $C=C(n)>0$ such that 
\begin{align}\label{eq:counting}
\sum_{Q}\int_{Q+Q}|\sg(u)(y)|\dif y\leq C\int_{\mathcal{N}_{\varepsilon}}|\sg(u)|\dif x. 
\end{align}
Indeed, for any $\widetilde{Q}\in\Gamma_{\varepsilon}$, there exist at most $k(n):=n^{3}+1$ cubes $\widetilde{Q}^{1},...,\widetilde{Q}^{k(n)}$ that touch $\widetilde{Q}$. Therefore, 
\begin{align*}
\sum_{Q}\int_{Q+Q}|\sg(u)(y)|\dif y & =\sum_{\widetilde{Q}}\sum_{j=1}^{k(n)}\int_{\widetilde{Q}^{j}}|\sg(u)(y)|\dif y=\sum_{j=1}^{k(n)}\sum_{\widetilde{Q}}\int_{\widetilde{Q}^{j}}|\sg(u)(y)|\dif y\\
& \leq \sum_{j=1}^{k(n)}\int_{\mathcal{N}_{\varepsilon}}|\sg(u)|\dif x\leq C\int_{\mathcal{N}_{\varepsilon}}|\sg(u)|\dif x
\end{align*}
and thus \eqref{eq:counting} follows. Collecting estimates, the proof is complete for $u\in (\hold^{1}\cap\ld)(\Omega)$. In the general case, we use the Reshetnyak continuity theorem \ref{thm:resh} in conjunction with the area--strict approximation lemma \ref{lem:area} to easily conclude the proof\footnote{Here it is important to use smooth approximation for the area--strict and \emph{not} the strict topology as Reshetnyak's Theorem only gives lower semicontinuity for the latter.}. 
\end{proof}
\begin{remark}
The inequality just proved needs to be put in the context of the usual Poincar\'{e} inequality, Lemma \ref{lem:poincare}. Without going into the above proof, suppose that the right side of \eqref{eq:poincare} vanishes so that $\sg(u)\equiv 0$ $\mathscr{L}^{n}$--a.e. and hence $u\in\mathcal{R}(N_{\varepsilon}(\Omega))$. Then the left side is zero too as convolutions with standard mollifiers leaves harmonic functions unchanged. On the other hand, it does not seem fruitful to write out the convolution with $\rho_{\varepsilon}$ since in this situation, 
\begin{align*}
|u(x)-\rho_{\varepsilon}*u(x)|=\left\vert\int_{\R^{n}}(u(x-y)-u(x))\rho_{\varepsilon}(y)\dif y\right\vert, 
\end{align*}
and in light of Ornstein's Non--Inequality, the attempt to bound the integral of this expression against the $\lebe^{1}$--norm of $\sg(u)$ is not clear to us provided we wish to produce $\varepsilon$ on the right side of \eqref{eq:poincare} instead of $\varepsilon^{s}$ with $0<s<1$. The latter corresponds to the embedding $\bd_{\locc}\hookrightarrow\sobo_{\locc}^{s,1}$ for $0<s<1$ as proved in \cite{GK}, Prop. 2.2, however, this seems to be not good enough for the partial regularity proof below. 
\end{remark}
A straightforward modification of the above arguments using Jensen's inequality gives the following
\begin{corollary}\label{cor:main}
Let $\widetilde{\Omega}\subset\R^{n}$ be open and suppose that $\Omega\Subset\tilde{\Omega}$ is a relatively compact subset and suppose $g\colon \R\to\R$ is a convex function of linear growth with $g(0)=0$ and which satisfies $g(\xi+\eta)\leq c(g(\xi)+g(\eta))$ for all $\xi,\eta\in\R$ with a fixed $c>0$. Let $\rho\in\hold^{\infty}(\ball(0,1);[0,1])$ be a standard mollifier. Then there exists a constant $C>0$ which only depends $\rho$ such that for all $0<\varepsilon<\dist(\Omega,\partial\widetilde{\Omega})$ and all $u\in\bd(\widetilde{\Omega})$ we have 
\begin{align*}
\int_{\Omega}g(|u-\rho_{\varepsilon}*u|)\dif x \leq C\varepsilon\,g(|\sg(u)|)(\overline{N_{\varepsilon}(\Omega)}), 
\end{align*}
where $g(|\sg(u)|)$ needs to be understood in the sense of convex functions of measures, see section \ref{sec:cfom}. 
\end{corollary}
\section{Partial Regularity}\label{sec:pr}
We now provide the proof of Theorem \ref{thm:1}, following the strategy outlined at the beginning of the present chapter. As an important convention, we put 
\begin{align}
X:=\begin{cases}
\bd &\;\text{if}\;p=1,\\
\sobo^{1,p} &\;\text{if}\;1<p<\infty
\end{cases}.
\end{align}
Moreover, let us arrange that if integrals of the form $\int_{U}f(\sg(u))\dif x$ with a convex function $f\colon\R_{\sym}^{n\times n}\to\R$ of linear growth and $u\in\bd(U)$ appear, then these integrals are tacitly understood in the sense of functions of measures, see section \ref{sec:cfom}. 
\subsection{The Excess}\label{sec:excess}
In the subsequent paragraphs we shall work with two different excess quantities. Keeping in mind that $1\leq p<\infty$ is fixed throughout, we put 
\begin{align}\label{eq:defe}
e(t):=(1+|t|^{2})^{p/2}-1,\;\;\;t\in\R
\end{align}
and define for $m\in\mathbb{N}$ the function $\mathbf{e}\colon \R^{m}\to\R_{0}^{+}$ for $\mathbf{A}\in\R^{M}$ by $\mathbf{e}(\mathbf{A}):=e(|\mathbf{A}|)$. Given $u\in X(\Omega)$, $z\in\Omega$ and $R>0$ such that $\ball(z,R)\subset\Omega$, we define 
\begin{align}
\excenew(u;z,R):=\int_{\ball(z,R)}\mathbf{e}(\sg(u)-(\sg(u))_{z,R})\dif x\;\;\;\text{and}\;\;\;\excesso(u;z,R):=\frac{\excenew(u;z,R)}{R^{n}}.  
\end{align}
For future reference we collect some preliminary estimates on $\mathbf{e}$. 
\begin{lemma}[\cite{AG}, Prop. 2.5]\label{lem:eest}
Let $1\leq p<\infty$ and define $\mathbf{e}$ by \eqref{eq:defe}. 
\begin{itemize}
\item[\emph{(i)}] There exists a positive constant $c=c(p)$ such that for all $\mathbf{A},\mathbf{B}\in\R^{M}$ and all $t>0$ there holds 
\begin{align*}
\mathbf{e}(t\mathbf{A})\leq c\,\max\{t^{2},t^{p}\}\mathbf{e}(\mathbf{A})\;\;\text{and}\;\;\mathbf{e}(\mathbf{A}+\mathbf{B})\leq c\,(\mathbf{e}(\mathbf{A})+\mathbf{e}(\mathbf{B})). 
\end{align*}
\item[\emph{(ii)}] Given $L>0$, there exists positive constants $C_{1}=C_{1}(L,p)$ and $C_{2}=C_{2}(L,p)$ such that for all $\mathbf{A}\in\R^{M}$ with $|\mathbf{A}|\leq L$ there holds 
\begin{align*}
C_{1}|\mathbf{A}|^{2}\leq \mathbf{e}(\mathbf{A}) \leq C_{2}|\mathbf{A}|^{2}. 
\end{align*}
\end{itemize}
\end{lemma}
\subsection{Decay of Comparison Maps}\label{sec:comp}
The purpose of this section is to provide decay estimates for H\"{o}lder continuous function which, in addition, satisfy a certain smallness condition. Later on in sections \ref{sec:molli} and \ref{sec:decay} we shall prove that suitable mollifications of minimisers match these conditions and thereby make the results of this section available. Throughout this paragraph we fix $\mathbf{m}\in\R_{\sym}^{n\times n}$, $\sigma>0$ and assume that $f\in \hold^{2}(\R_{\sym}^{n\times n})$ satisfies 
\begin{align}
\lambda|\mathbf{Z}|^{2}\leq \langle f''(\mathbf{m})\mathbf{Z},\mathbf{Z}\rangle \leq\Lambda|\mathbf{Z}|^{2}. 
\end{align}
for some $0<\lambda\leq\Lambda<\infty$ and all $\mathbf{Z}\in E$. Moreover, we assume that there exists a bounded and non--decreasing function $\omega\colon \R_{0}^{+}\to\R_{0}^{+}$ with $\lim_{t\searrow 0}\omega(t)=0$ such that for all $\mathbf{m}'\in\mathbb{B}(\mathbf{m},\sigma)$ we have 
\begin{align}\label{eq:modcon}
|f''(\mathbf{m})-f''(\mathbf{m}')|\leq\omega(|\mathbf{m}-\mathbf{m}'|). 
\end{align}
Finally, for $0<r<R$ and $v\in \hold^{1,\alpha}(\overline{\ball(z,r)};\R^{n})$ we put 
\begin{align*}
&\devi(v;z,r):=\int_{\ball(z,r)}f(\sg(v))-\inf\left\{\int_{\ball(z,r)}f(\sg(w))\dif x\colon\begin{array}{c}
w\in \hold^{1,\alpha}(\ball(z,r);\R^{n}) \\ w = v\;\text{on}\;\partial\ball(z,r)
\end{array}\right\},\\
&\mathbf{t}(v;z,r):=\sup_{\ball(z,r)}|\sg(v)-\mathbf{m}|+2^{\alpha}r^{\alpha}[\sg(v)]_{C^{0,\alpha}(\ball(z,r);\R^{n\times n})}
\end{align*}
Notice that $\devi$ is an indicator of how far $v$ is away from minimising the variational integral $\mathscr{F}$ restricted to $\ball(z,r)$. Besides, the function $\mathbf{t}$ will prove useful to find the mentioned smallness condition which is necessary to infer the decay estimate of the H\"{o}lder continuous comparison maps. Now we have 
\begin{proposition}\label{prop:1}
Fix $0<\alpha<1$. Then there exists $c_{1}>0$ such that the following holds: If $v\in \hold^{1,\alpha}(\ball(z,R/2);\R^{n})$ satisfies $\mathbf{t}(R/2)<\min\{\sigma/c_{1},1\}$, then there exists a bounded, non--decreasing function $\vartheta\colon\R_{0}^{+}\to\R_{0}^{+}$ with $\lim_{t\searrow 0}\vartheta(t)=0$ such that for all $0<r<R/2$ we have 
\begin{align}\label{eq:smoothes}
\begin{split}
\int_{\ball(z,r)}|\sg(v)-(\sg(v))_{z,r}|^{2}\dif x &\lesssim \left(\frac{r}{R}\right)^{n+2}\int_{\ball(z,R/2)}|\sg(v)-(\sg(v))_{z,R/2}|^{2}\dif x \\
& +\vartheta(\mathbf{t}(R/2))\int_{\ball(z,R/2)}|\sg(v)-\mathbf{m}|^{2}\dif x \\ & + \devi(v;z,R/2)
\end{split}
\end{align}
provided $1\leq p\leq 2$ and 
\begin{align}\label{eq:smoothes1}
\begin{split}
\int_{\ball(z,r)}|\sg(v)-(\sg(v))_{z,r}|^{p}\dif x &\lesssim \left(\frac{r}{R}\right)^{n+p}\int_{\ball(z,R/2)}|\sg(v)-(\sg(v))_{z,R/2}|^{p}\dif x \\
& +\vartheta(\mathbf{t}(R/2))\int_{\ball(z,R/2)}|\sg(v)-\mathbf{m}|^{2}\dif x \\ & + \devi(v;z,R/2)
\end{split}
\end{align}
provided $p\geq 2$. 
\end{proposition}
\begin{proof}
The strategy of the proof is as follows: First, passing to the second order Taylor polynomial of the integrand $f$ we obtain an integrand $g$ of quadratic growth to whose minimisers we may apply the classical decay estimates for linear elliptic systems. Then we show that the conditions of the present proposition are sufficient for these decay estimates to carry over to $v$ in a way such that \eqref{eq:smoothes} follows. Moreover, we shall only prove Eq. \eqref{eq:smoothes} only; Eq. \eqref{eq:smoothes1} follows in the same way, and we will indicate where slight changes must be incorporated. 

We therefore begin by defining the auxiliary integrand $g\colon\mathbb{B}(\mathbf{m},\sigma)\to\R$ through 
\begin{align*}
g(\mathbf{Z}):=f(\mathbf{m})+\langle f'(\mathbf{m}),(\mathbf{Z}-\mathbf{m})\rangle+\frac{1}{2}\langle f''(\mathbf{m})(\mathbf{Z}-\mathbf{m}),(\mathbf{Z}-\mathbf{m})\rangle,\;\;\;\mathbf{Z}\in\mathbb{B}(\mathbf{m},\sigma).  
\end{align*}
Using a Taylor expansion of $f$ up to order two around $\mathbf{m}$, we deduce by Eq. \eqref{eq:modcon} that 
\begin{align}\label{eq:Taylorsecond}
|f(\mathbf{Z})-g(\mathbf{Z})|\lesssim \omega(|\mathbf{Z}-\mathbf{m}|)|\mathbf{Z}-\mathbf{m}|^{2},\;\;\;\mathbf{Z}\in\mathbb{B}(\mathbf{m},\sigma). 
\end{align}
By Lemma \ref{lem:lineardecayestimates}, $h\in \hold^{1,\alpha}(\ball(z,R/2);\R^{n})$ of the auxiliary minimisation problem 
\begin{align}\label{eq:auxprob}
\text{to minimise}\;\int_{\ball(z,R/2)}g(\sg(w))\dif x\;\;\text{among all}\;w\in W_{v}^{1,2}(\ball(z,R/2);\R^{n}), 
\end{align}
where $W_{v}^{1,2}(\ball(z,R/2);\R^{n})=v+\sobo_{0}^{1,2}(\ball(z,R/2);\R^{n})$. In consequence, we have
\begin{align}\label{eq:compare}
\int_{\ball(z,r)}|\adop (h)-(\adop (h))_{z,r}|^{2}\dif x \lesssim \left(\frac{r}{R}\right)^{n+2}\int_{\ball(z,R/2)}|\adop (h)-(\adop (h))_{z,R/2}|^{2}\dif x
\end{align}
for all $0<r<R/2$. Moreover, Lemma \ref{lem:lineardecayestimates}~(b) gives
\begin{align}\label{eq:Schauder}
[\sg(h)]_{\hold^{0,\alpha}(\ball(z,R/2);\R^{n\times n})}\lesssim [\sg(v)]_{\hold^{0,\alpha}(\ball(z,R/2);\R^{n\times n})}.
\end{align}
We will now compare $v$ with $h$. To this end, we first notice that
\begin{align*}
\int_{\ball(z,r)}|\sg(v)-(\sg(v))_{z,r}|^{2}\dif x & \leq \int_{\ball(z,r)}|\sg(v)-\sg(h)|^{2}\dif x \\ 
& + \int_{\ball(z,r)}|\sg(h)-(\sg(h))_{z,r}|^{2}\dif x \\
& + \int_{\ball(z,r)}|(\sg(h))_{z,r}-(\sg(v))_{z,r}|^{2}\dif x =:\mathbf{I}+\mathbf{II}+\mathbf{III}.
\end{align*}
We shall estimate $\mathbf{II}$ through \eqref{eq:compare}. Keeping this in mind, we turn to the remaining two terms $\mathbf{I}$ and $\mathbf{III}$. We have 
\begin{align*}
\mathbf{I}+\mathbf{III}& \lesssim  \int_{\ball(z,r)}|\sg(v)-\sg(h)|^{2}\dif x +  \int_{\ball(z,r)}\dashint_{\ball(z,r)}|\sg(v)-\sg(h)|^{2}\dif y \dif x\\
& \lesssim \int_{\ball(z,r)}|\sg(v)-\sg(h)|^{2}\dif x, 
\end{align*}
and by the minimality property of $h$ we deduce through Eq. \eqref{eq:compare} that 
\begin{align*}
\mathbf{II} &\lesssim \left(\frac{r}{R}\right)^{n+2}\int_{\ball(z,R/2)}|\adop (h)-(\adop (h))_{z,R/2}|^{2}\dif x \\ & \lesssim \left(\frac{r}{R}\right)^{n+2}\int_{\ball(z,R/2)}|\adop (v)-(\adop (v))_{z,R/2}|^{2}\dif x. 
\end{align*}
At this stage we note that the case $p\geq 2$ is covered by replacing estimate \eqref{eq:compare} by 
\begin{align*}
\int_{\ball(z,r)}|\adop (h)-(\adop (h))_{z,r}|^{p}\dif x \lesssim \left(\frac{r}{R}\right)^{n+p}\int_{\ball(z,R/2)}|\adop (h)-(\adop (h))_{z,R/2}|^{p}\dif x
\end{align*}
for all $0<r<R/2$; the rest of the proof then goes along the same lines as for $1\leq p\leq 2$. Turning back to $1\leq p\leq 2$ and combining the estimates for $\mathbf{I},\mathbf{II}$ and $\mathbf{III}$, we need to estimate 
\begin{align*}
\int_{\ball(z,R/2)}|\sg(v)-\sg(h)|^{2}\dif x & \lesssim \int_{\ball(z,R/2)}\langle f''(\mathbf{m})(\sg(v)-\sg(h)),(\sg(v)-\sg(h))\rangle\dif x \\ 
& \lesssim \int_{\ball(z,R/2)}g(\sg(v))-f(\sg(v))\dif x  \\ &+ \int_{\ball(z,R/2)}f(\sg(v))-f(\sg(h))\dif x \\
& + \int_{\ball(z,R/2)}f(\sg(h))-g(\sg(h))\dif x =: \mathbf{I}_{1}+\mathbf{I}_{2}+\mathbf{I}_{3}. 
\end{align*}
The integrals $\mathbf{I}_{1},\mathbf{I}_{2},\mathbf{I}_{3}$ will be estimated separately. Among them, we readily obtain through the definition of $\devi$ and $h$ that $\mathbf{I}_{2}\leq \devi(v;z,R/2)$. Now notice that $h$ solves the auxiliary problem \eqref{eq:auxprob} and hence we have that $\|\sg(h)-\mathbf{m}\|_{\lebe^{2}(\ball(z,R/2);\R^{n\times n})} \lesssim \|\sg(v)-\mathbf{m}\|_{\lebe^{2}(\ball(z,R/2);\R^{n\times n})}$. Therefore we deduce that 
\begin{align*}
\sup_{\ball(z,R/2)}|\sg(h)-\mathbf{m}| & \leq \sup_{\ball(z,R/2)}|\sg(h)-(\sg(h))_{z,R/2}| + \sup_{\ball(z,R/2)}|(\sg(h))_{z,R/2}-\mathbf{m}|\\
& \lesssim R^{\alpha}[\sg(h)]_{\hold^{0,\alpha}(\ball(z,R/2);\R^{Nn})}+\left(\dashint_{\ball(z,R/2)}|\sg(h)-\mathbf{m}|^{2}\dif x \right)^{\frac{1}{2}}\\
& \lesssim R^{\alpha}[\sg(v)]_{\hold^{0,\alpha}(\ball(z,R/2);\R^{Nn})}+\left(\dashint_{\ball(z,R/2)}|\sg(v)-\mathbf{m}|^{2}\dif x \right)^{\frac{1}{2}}\\
& \lesssim R^{\alpha}[\sg(v)]_{\hold^{0,\alpha}(\ball(z,R/2);\R^{Nn})}+\sup_{\ball(z,R/2)}|\sg(v)-\mathbf{m}|\\
& =: c_{0}\mathbf{t}(v;z,R/2),
\end{align*}
where $c_{0}=c_{0}(\lambda,\Lambda,n,N,M)>0$ is a constant. At this point we make our definition of $c_{1}>0$ as it appears in the assumptions of the present proposition by putting $c_{1}:=c_{0}$. Then, by assumption, we have that $\sg(h)(x)\in\mathbb{B}(\mathbf{m},\sigma)$ provided $x\in \ball(z,R/2)$ and so $|f(\sg(h)(x))-g(\sg(h)(x))|\lesssim \omega(|\sg(h)(x)-\mathbf{m}|)|\sg(h)(x)-\mathbf{m}|^{2}$ for all $x\in\ball(z,R/2)$ by Eq. \eqref{eq:Taylorsecond}. In consequence,
\begin{align*}
\int_{\ball(z,R/2)}f(\sg(h))-g(\sg(h))\dif x \lesssim \omega(c_{0}\mathbf{t}(v;z,R/2))\int_{\ball(z,R/2)}|\sg(h)-\mathbf{m}|^{2}\dif x. 
\end{align*}
The remaining integral $\mathbf{I}_{3}$ is estimated as follows: By assumption, we have the estimate $\mathbf{t}(v;z,R/2)<\min\{\sigma/c_{1},1\}$ so that it holds $\sg(v)(x)\in\mathbb{B}(\mathbf{m},\sigma)$ for all $x\in\ball(z,R/2)$. Referring to eq. \eqref{eq:Taylorsecond} we then conclude 
\begin{align*}
\int_{\ball(z,R/2)}g(\sg(v))-f(\sg(v))\dif x \leq \omega(\mathbf{t}(v;z,R/2))\int_{\ball(z,R/2)}|\sg(v)-\mathbf{m}|^{2}\dif x. 
\end{align*}
Collecting estimates, the claim follows. 
\end{proof}
\subsection{Mollification}\label{sec:molli}
In the following, let $\eta\in \hold_{c}^{\infty}(\R^{n};[0,1])$ denote the standard mollifier 
\begin{align*}
\eta(x):=\tilde{c}\mathbbm{1}_{\ball(0,1)}(x)\exp\left(-\frac{1}{|x|^{2}-1}\right),\;\;\;x\in\R^{n}
\end{align*}
so that $\spt(\eta)\subset\overline{\ball(0,1)}$, with a constant $\tilde{c}>0$ such that $\|\eta\|_{L^{1}}=1$. Given $u\in L_{\locc}^{1}(\Omega;\R^{N})$ and $\delta,\varepsilon>0$ we define the mollifications $u_{\delta}$ and $u_{\varepsilon}$ by 
\begin{align*}
u_{\delta}(x):=\frac{1}{\delta^{n}}\int_{\ball(x,\delta)}\eta\left(\frac{x-y}{\delta}\right)u(y)\dif y,\;\;\;u_{\varepsilon}(x):=\dashint_{\ball(x,\varepsilon)}u(y)\dif y, 
\end{align*}
so that $u_{\delta}$ is defined on $\{x\in\Omega\colon\,\dista(x,\partial\Omega)>\delta\}$ and $u_{\varepsilon}$ on $\{x\in\Omega\colon\,\dista(x,\partial\Omega)>\varepsilon\}$. Finally, we put 
$u_{\delta,\varepsilon}:=(u_{\varepsilon})_{\delta}$ which, in consequence, is defined on $\{x\in\Omega\colon\;\dista(x,\partial\Omega)>\delta+\varepsilon\}$. In the next step, we prove that the smallness of the excess guarantees that $\delta$ and $\varepsilon$ can be adjusted in a way such that Proposition \ref{prop:1} applies to $u_{\delta,\varepsilon}$. 
\begin{lemma}\label{lem:adjust}
Let $u\in X(\ball(x_{0},r))$ and put $\mathbf{m}:=(\sg(u))_{\ball(x_{0},r)}$. Fix $0<\alpha<1$ and suppose that $\excesso(u;x_{0},r)<1$. Then there exist $0<\gamma<\frac{1}{n+2\alpha}$ and $\beta>0$ such that if 
\begin{align}\label{eq:deltaeps}
\delta = \varepsilon = \frac{1}{16}r\excesso(u;x_{0},r)^{\gamma}, 
\end{align}
then $\mathbf{t}(u;x_{0},r/2)\lesssim \excesso(u;x_{0},r)^{\beta}$. 
\end{lemma}
\begin{proof}
First observe that as a consequence of the elementary estimates for convolutions we obtain 
\begin{align}\label{eq:step}
\mathbf{t}(u_{\delta,\varepsilon};x_{0},\frac{r}{2}) & \leq c\,\left(1+\left(\frac{r}{\delta}\right)^{\alpha}\right)\sup_{x\in\ball(x_{0},\frac{r}{2}+\delta)}|\sg(u_{\varepsilon})-\mathbf{m}|.\end{align}
Now recall that $\varepsilon$ and $\delta$ are adjusted according to \eqref{eq:deltaeps}. Then, by Jensen's inequality and a change of variables we deduce 
\begin{align*}
\mathbf{e}(\sg(u_{\varepsilon})(x)-\mathbf{m}) & \leq \dashint_{\ball(x,\varepsilon)}\mathbf{e}(\sg(u)-\mathbf{m})\dif y = \left(\frac{r}{\varepsilon}\right)^{n}\excesso(u;x_{0},r)\\
& \leq 16^{n}\excesso(u;x_{0},r)^{1-n\gamma} \leq 16^{n}, 
\end{align*}
the last estimate being valid due to our assumption $\excesso(u;x_{0},r)<1$. However, for any $K>0$ there exists a constant $c=c(K)$ such that if $|\mathbf{A}|\leq K$, then $|\mathbf{A}|^{2}\leq c\,\mathbf{e}(\mathbf{A})$. Applying this with $K=16^{n}$, we obtain for all $x\in\ball(x_{0},\frac{\delta}{2}+r)$ that
\begin{align*}
|\sg(u_{\varepsilon})(x)-\mathbf{m}|^{2} \lesssim \mathbf{e}(\sg(u_{\varepsilon})(x)-\mathbf{m})\lesssim \excesso(u;x_{0},r)^{1-n\gamma}. 
\end{align*}
Now observe that by \eqref{eq:step} we have by the specific choice of $\delta=\frac{1}{16}r\excesso(u;x_{0},r)^{\gamma}$ that 
\begin{align*}
\mathbf{t}(u_{\delta,\varepsilon};x_{0},r/2) & \leq c\,\left(1+\left(\excesso(u;x_{0},r)\right)^{-\gamma\alpha}\right)\excesso(u;x_{0},r)^{\frac{1-n\gamma}{2}}\lesssim \excesso(u;x_{0},r)^{\beta}
\end{align*}
with $\beta:=\frac{1}{2}(1-n\gamma)-\gamma\alpha>0$ (recall that $0<\gamma<1/(n+2\alpha)$). This proves the claim.  
\end{proof}
\begin{corollary}\label{cor:adjust}
In the situation of Lemma \ref{lem:adjust}, we have 
\begin{align*}
&\int_{\ball(x_{0},r/2)}|\sg(u_{\delta,\varepsilon})-\mathbf{m}|^{2}\dif x \lesssim \excenew(u;x_{0},r),\;\;\;\text{and}\\&\int_{\ball(x_{0},r/2)}|\sg(u_{\delta,\varepsilon})-(\sg(u_{\delta,\varepsilon}))_{\ball(x_{0},r/2)}|^{2}\dif x \lesssim \excenew(u;x_{0},r).
\end{align*}
\end{corollary}
\begin{proof}
We recall from Lemma \ref{lem:adjust} that $\mathbf{t}(u_{\delta,\varepsilon};x_{0},r/2)\leq C \excesso(u;x_{0},r)^{\beta}\leq C$. In particular, there holds $\sup_{x\in\ball(x_{0},r/2)}|\sg(u_{\delta,\varepsilon})-\mathbf{m}|\leq C$. Therefore we have $|\sg(u_{\delta,\varepsilon})(x)-\mathbf{m}|^{2}\leq c\,\mathbf{e}(\sg(u_{\delta,\varepsilon})-\mathbf{m})$ with a constant $c>0$ which only depends on $C$ and $\mathbf{e}$ for every $x\in\ball(x_{0},r/2)$. We conclude by Jensen's Inequality that
\begin{align*}
\int_{\ball(x_{0},r/2)}|\sg(u_{\delta,\varepsilon})-\mathbf{m}|^{2}\dif x & \lesssim \int_{\ball(x_{0},r/2)}\mathbf{e}(\sg(u_{\delta,\varepsilon})-\mathbf{m}) \\
& \lesssim \int_{\ball(x_{0},r/2)}(\mathbf{e}(\sg(u)-\mathbf{m}))_{\delta,\varepsilon}\dif x = \excenew(u;x_{0},r), 
\end{align*}
the last estimate being valid due to $\frac{r}{2}+\delta+\varepsilon<r$ by assumption. The second inequality directly follows from this and hence the complete statement of the corollary. 
\end{proof}
\subsection{Decay Estimate}\label{sec:decay}
Having proved suitable decay estimates for smooth comparison maps, we turn now to the proof of how these inherit to the actual minimiser. This is accomplished by introducing a new variational integrand, essentially the first order Taylor approximation of the integrand $f$, and studying a closely related minimisation problem for which the comparison argument can be carried out conveniently. Combining the results of this section with the decay estimates of the mollified minimisers given in the preceding section, we will be in position to deduce the decay estimate for the minimiser itself.

To do so, we first make some definitions. Let $\mathbf{m}:=(\sg(u))_{\ball(z,R)}$ and $\sigma>0$. For $\mathbf{m}'\in\mathbb{B}(\mathbf{m},\sigma)$ we define a new convex integrand $\tilde{f}$ by 
\begin{align}\label{eq:tildef}
\tilde{f}(\mathbf{m}''):=f(\mathbf{m}'+\mathbf{m}'')-f(\mathbf{m}')-\langle f'(\mathbf{m}'),\mathbf{m}''\rangle,\;\;\;\mathbf{m}''\in \R_{\sym}^{n\times n}.
\end{align}
For further reference we record that $\tilde{f}$ is convex as well and satisfies both $\tilde{f}(0)=0$ and $\tilde{f}'(0)=0$. Moreover, since $f\approx\mathbf{e}$, we also have $\tilde{f}\approx\mathbf{e}$. Given $w\colon \ball(z,R)\to \R^{n}$ and $\mathbf{m}'\in \R_{\sym}^{n\times n}$, we define $\tilde{w}\colon\ball(z,R)\to \R^{n}$ by 
\begin{align}\label{eq:tildew}
\tilde{w}(x):=w(x)-\mathbf{m}'(x-z).
\end{align} 
We shall now provide three auxiliary statements which occupy a central position in the proof of the decay estimate. However, we stress that the following preparatory lemmas do \emph{not} take into account minimality of $u$ but strictly rest on the convexity assumption imposed on $f$. In particular, this is the case for Lemma \ref{lem:convex} whose generalisation to, e.g., quasiconvex functionals lacks a proof. 
\begin{lemma}
Let $u,v\in X(\ball(z,R))$ such that $u=v$ on $\partial\!\ball(z,R)$ in the sense of traces and define $\tilde{f}$ and $\tilde{u},\tilde{v}$ by \eqref{eq:tildef} and \eqref{eq:tildew}, respectively. Then there holds 
\begin{align*}
\int_{\ball(z,R)}f(\sg(u))-f(\sg(v))\dif x = \int_{\ball(z,R)}\tilde{f}(\sg(\tilde{u}))-\tilde{f}(\sg(\tilde{v}))\dif x
\end{align*}
\end{lemma}
\begin{proof}
Writing out the definitions, we obtain 
\begin{align*}
\int_{\ball(z,R)}\tilde{f}(\sg(\tilde{u}))-\tilde{f}(\sg(\tilde{v}))\dif x & = \int_{\ball(z,R)}f(\mathbf{m}'+\sg(\tilde{u}))- f(\mathbf{m}'+\sg(\tilde{v}))\dif x \\
& - \int_{\ball(z,R)} f'(\mathbf{m}')\sg(\tilde{u})-f'(\mathbf{m}')\sg(\tilde{v})\dif x \\
& = \int_{\ball(z,R)}f(\sg(u))- f(\sg(v))\dif x \\
& - \int_{\ball(z,R)} f'(\mathbf{m}')\sg(\tilde{u})- f'(\mathbf{m}')\sg(\tilde{v})\dif x. 
\end{align*}
Now notice that $\sg(\tilde{u})-\sg(\tilde{v})=\sg(u)-\sg(v)$ and so, by the usual Gauss--Green Theorem if $1<p<\infty$ or \eqref{eq:trace} if $p=1$,
\begin{align*}
\int_{\ball(z,R)} f'(\mathbf{m}')(\sg(\tilde{u})-\sg(\tilde{v}))\dif x & = \int_{\partial\ball(z,R)}\langle f'(\mathbf{m}'),\trace(u-v)\odot\nu)\dif\mathscr{H}^{n-1} \\
& - \int_{\ball(z,R)}\di((f'(\mathbf{m}')))(u-v)\dif x = 0
\end{align*}
because $\di((f'(\mathbf{m}')))=0$ and, by our assumptions on $u$ and $v$, $\trace(u)=\trace(v)$ $\mathcal{H}^{n-1}$--a.e. on $\partial\ball(z,R)$. The claim follows. 
\end{proof}
\begin{lemma}\label{lem:convex}
Let $u\in X(\Omega)$ and let $\varepsilon,\delta,R>0$ and $z\in\Omega$ be such that $\ball(z,R)\Subset\Omega$. Then the following holds: 
\begin{itemize}
\item[\emph{(i)}] If $0<t_{1}<t_{2}<R-\varepsilon-\delta$, then there exists $t\in (t_{1},t_{2})$ such that 
\begin{align}\label{eq:convexo}
\int_{\ball(z,t)}f(\sg(u_{\varepsilon,\delta})-f(\sg(u)) \lesssim \frac{\varepsilon+\delta}{t_{2}-t_{1}}\int_{\ball(z,R)}f(\sg(u)). 
\end{align}
\item[\emph{(ii)}] If $0<t_{1}<t_{2}<R-\varepsilon-\delta$, $0<r<R/4$ and $R/2 \leq t_{1} <R-\varepsilon-\delta$, then there exist $r'\in (r,2r)$ and $t'\in (t_{1},t_{2})$ with 
\begin{align}
\int_{\ball(z,t')\setminus\ball(z,r')}f(\sg(u_{\varepsilon,\delta}))-f(\sg(u))\lesssim \left(\frac{\varepsilon+\delta}{t_{2}-t_{1}}+\frac{\varepsilon+\delta}{r}\right)\int_{\ball(z,R)}f(\sg(u)). 
\end{align}
\end{itemize}
\end{lemma}
\begin{proof}
Following \cite{AG}, the proof of (i) or (ii), respectively, is accomplished by considering the auxiliary functions $g,h\colon \R^{n}\to\R$ given by 
\begin{align*}
&g(x):=(t_{2}-t_{1})\mathbbm{1}_{\{|x|<t_{1}\}}(x)+(t_{2}-|x|)\mathbbm{1}_{\{t_{1}\leq |x|\leq t_{2}\}}(x),\;x\in\R^{n} 
\end{align*}
for the first part of the lemma and 
\begin{align*}
h(x)&:=\frac{(t_{2}-t_{1})(|x|-r)}{r}\mathbbm{1}_{\{r\leq|x|<2r\}}(x)+(t_{2}-t_{1})\mathbbm{1}_{\{2r\leq |x|<t_{1}\}}(x)\\ &+(t_{2}-|x|)\mathbbm{1}_{\{t_{1}\leq |x|<t_{2}\}}(x),\;x\in\R^{n}
\end{align*}
for the second. In the following we will concentrate on (i) and consequently the use of $g$; (ii) follows in the same way. The strategy of the proof is to first define
\begin{align*}
\mathtt{Fail}:=\{t\in (t_{1},t_{2})\colon\;\text{inequality}\;\eqref{eq:convexo}\;\text{is not satisfied by}\;t\}
\end{align*}
and then to show that $\mathscr{L}^{1}(\mathtt{Fail})\leq \tfrac{1}{2}(t_{2}-t_{1})$. From this it follows that the one--dimensional Lebesgue measure of the set where \eqref{eq:convexo} holds is at least $\tfrac{1}{2}(t_{2}-t_{1})$ and so the set must be non--empty; this will imply (i). To prove the estimate for $\mathtt{Fail}$, we first observe by the definition of $g$ that if $t_{1}\leq t \leq t_{2}$, then $\ball(z,t)=\{x\in\ball(z,R)\colon t<g(x)\}$. So we have for any $\psi\in L^{1}(\ball(z,R);\R)$ that 
\begin{align*}
\int_{t_{1}}^{t_{2}}\int_{\ball(z,t)}\psi(y)\dif y \dif t & = \int_{t_{1}}^{t_{2}}\int_{\{x\in\ball(z,R)\colon t<g(x)\}}\psi(y)\dif y \dif t \\
& = \int_{0}^{\infty}\int_{\{x\in\ball(z,R)\colon t<g(x)\}}\psi(y)\dif y \dif t = \int_{\ball(z,R)}\psi(y)g(y)\dif y. 
\end{align*}
Now recall the notation $\psi_{\delta},\psi_{\varepsilon}$ and $\psi_{\delta,\varepsilon}$ from section \ref{sec:molli}. Choosing $\delta$ and $\varepsilon$ due to the assumptions of the present lemma, we obtain 
\begin{align*}
\int_{t_{1}}^{t_{2}}\int_{\ball(z,t)}(\psi-\psi_{\delta})\dif x\dif t & = \int_{\ball(z,R)}(\psi-\psi_{\delta})g\dif x  = \int_{\ball(z,R)}(g-g_{\delta})\psi\dif x \\
& \leq \sup_{\ball(z,R)}|g_{\delta}-g|\int_{\ball(z,t_{2}+\delta)}\psi\,\dif x \leq \delta \int_{\ball(z,t_{2}+\delta)}\psi\,\dif x
\end{align*}
and by the same estimation we also get 
\begin{align*}
\int_{t_{1}}^{t_{2}}\int_{\ball(z,t)}(\psi-\psi_{\varepsilon})\dif x\dif t \leq \varepsilon \int_{\ball(z,t_{2}+\varepsilon)}\psi\,\dif x. 
\end{align*}
Combined with $f(\sg(u_{\delta,\varepsilon}))\leq(f(\sg(u_{\varepsilon}))_{\delta}$ which holds due to Jensen's inequality by convexity of $f$, these estimates imply by $t_{2}+\delta+\varepsilon<R$ that 
\begin{align*}
\int_{t_{1}}^{t_{2}}\int_{\ball(z,t)}f(\sg(u_{\delta,\varepsilon}))-f(\sg(u))\dif x \dif t & \leq \int_{t_{1}}^{t_{2}}\int_{\ball(z,t)}(f(\sg(u_{\varepsilon})))_{\delta}-f(\sg(u_{\varepsilon}))\dif x \dif t \\
& + \int_{t_{1}}^{t_{2}}\int_{\ball(z,t)}f(\sg(u_{\varepsilon}))-f(\sg(u))\dif x \dif t \\
& \leq (\delta+\varepsilon)\int_{\ball(z,R)}f(\sg(u))\dif x. 
\end{align*}
Now, if $t'\in\mathtt{Fail}$, then 
\begin{align*}
\int_{t_{1}}^{t_{2}}\int_{\ball(z,t)}f(\sg(u_{\delta,\varepsilon}))-f(\sg(u))\dif x \dif t \leq \frac{t_{2}-t_{1}}{2}\int_{\ball(z,t')}f(\sg(u_{\delta,\varepsilon}))-f(\sg(u))\dif x, 
\end{align*}
so integrating over all $t'\in\mathtt{Fail}$ an regrouping terms yields 
\begin{align*}
\mathscr{L}^{1}(\mathtt{Fail})\leq \frac{t_{2}-t_{1}}{2}. 
\end{align*}
As we explained, this implies (i). Assertion (ii) is proved in the same way and so the claim follows. 
\end{proof}
The main difference to the case $p=1$ is given in the following lemma, where Korn's Inequality plays an essential role. In the following, given $z\in\R^{n}$ and $0<t<s$, we shall use the shorthand
\begin{align*}
A(z;t,s):=\{x\in\R^{n}\colon\;t<|x-z|<s\}
\end{align*}
for the annulus centered at $z$ with inner radius $t$ and outer radius $s$. 
\begin{lemma}\label{lem:molest}
Let $1\leq p<\infty$ and let $u\in X(\Omega)$. Let $x_{0}\in\Omega$ and choose $r,\varepsilon,\delta,s,t,L>0$ such that $r/2 < t < s < r$, $s+\delta+\varepsilon<r$ and $\ball(x_{0},r)\subset\Omega$. Then there exists a constant $c>0$ such that 
\begin{align}\label{eq:critical}
\int_{A(x_{0};t,s)}\mathbf{e}(L\;(u-u_{\delta,\varepsilon}))\dif x \leq c\,\max\{(L\varepsilon)^{p},(L\varepsilon)^{2}\}\int_{\overline{A(x_{0};t-\varepsilon,s+\varepsilon)}}\mathbf{e}(\sg(u))\dif x. 
\end{align}
\end{lemma}
\begin{proof}
The estimate is an immediate consequence of Corollary~\ref{cor:main} in conjunction with Lemma \ref{lem:eest}~(i).
\end{proof}
The next two propositions form the main part of the partial regularity proof. Thus we stress that from now on we are working with the fact that $u\in X(\Omega)$ is a local ($\bd$--)minimiser of the functional \eqref{eq:varprin}. 
\begin{proposition}\label{prop:main1}
Let $f\colon \R_{\sym}^{n\times n}\to\R$ be a convex $\hold^{2}$--function of $p$--growth, $1\leq p<\infty$. Then there exists $a>0$ such that for any minimiser $u\in X(\Omega)$ the following holds: If we have 
\begin{align*}
(\sg(u))_{z,R}=\mathbf{m}\;\;\;\text{and}\;\;\;\excesso(u;z,R)<a,
\end{align*}
and there exists $\sigma>0$ such that $f$ is of class $\hold^{2}$ in $\mathbb{B}(\mathbf{m},\sigma)$ verifying
\begin{align*}
|f''(\mathbf{m}')-f''(\mathbf{m})|\leq\omega(|\mathbf{m}-\mathbf{m}'|)\;\;\;\text{for all}\;\mathbf{m},\mathbf{m}'\in\mathbb{B}(\mathbf{m},\sigma)
\end{align*}
with a bounded, non--decreasing function $\omega\colon\R_{0}^{+}\to\R_{0}^{+}$ such that we have 
\begin{align*}
\lambda |\mathbf{Z}|^{2}\leq \langle f''(\mathbf{m})\mathbf{Z},\mathbf{Z}\rangle\leq \Lambda|\mathbf{Z}|^{2}
\end{align*}
for two constants $0<\lambda\leq\Lambda<\infty$ and all $\mathbf{Z}\in \R_{\sym}^{n\times n}$, then there exists a constant $c>0$ and a bounded and non--decreasing function $h\colon\R_{0}^{+}\to\R_{0}^{+}$ with $\lim_{t\searrow 0}h(t)=0$ such that
\begin{align}\label{eq:partial}
\excenew(u;z,r)\leq c\excenew(v;z,2r)+ch(\excesso(u;z,R))(1+(R/r)^{n+1})\excenew(u;z,R)
\end{align}
holds for all $0<r<R/4$. Here we have set $v:=u_{\delta,\varepsilon}$ with $\delta,\varepsilon$ adjusted according to Lemma \ref{lem:adjust}. 
\end{proposition}
\begin{proof}
Let $0<r<R/4$ and put $\mathbf{m}':=(\sg(v))_{\ball(z,r)}$. Due to our choice of $\delta$ and $\varepsilon$, Lemma \ref{lem:adjust} gives us $|\mathbf{m}-\mathbf{m}'|<c_{1}\exce^{\beta}(u;z,R)$ and thus we may choose $0<a<\frac{1}{2}$ sufficiently small such that $\exce(u;z,R)<a$ yields $|\mathbf{m}-\mathbf{m}'|<\sigma/2$. We observe that 
\begin{align}\label{eq:usmall}
\begin{split}
\excenew(u;z,r) & = \int_{\ball(z,r)}\mathbf{e}(\sg(u)-(\sg(u))_{\ball(z,r)}) \\
& \lesssim  \int_{\ball(z,r)}\mathbf{e}(\sg(u)-\mathbf{m}')+ \int_{\ball(z,r)}\mathbf{e}((\sg(u))_{\ball(z,r)}-\mathbf{m}')\\
& \lesssim \int_{\ball(z,r)}\mathbf{e}(\sg(u)-\mathbf{m}') \lesssim \int_{\ball(z,r)}\tilde{f}(\sg(\tilde{u})). 
\end{split}
\end{align}
Now fix $R/2<t<s<\tfrac{7}{8}R$ and define $\rho\in \hold_{c}(\ball(z,R);[0,1])$ by 
\begin{align}\label{eq:cutoff}
\rho(x):=\frac{2}{s-t}(|x|-t)\mathbbm{1}_{\{t\leq |x|\leq (s+t)/2\}}(x)+\mathbbm{1}_{\{|x|>(s+t)/2\}}(x),\;\;\;x\in\ball(z,R). 
\end{align}
Then we have $\tilde{v}+\rho(\tilde{u}-\tilde{v})\in X(\Omega)$. Using that $\tilde{u}$ is a local minimiser of the modified functional $\tilde{\mathscr{F}}$ in the second step, we deduce that 
\begin{align*}
\int_{\ball(z,r)}\tilde{f}(\sg(\tilde{u}))\dif x & = \left(\int_{\ball(z,r)} +\int_{\ball(z,t)\setminus\ball(z,r)}+\int_{\ball(z,s)\setminus\ball(z,t)}\right)\tilde{f}(\sg(\tilde{u}))\dif x \\ 
& \leq \int_{\ball(z,r)}\tilde{f}(\sg(\tilde{v}))\dif x + \int_{\ball(z,t)\setminus\ball(z,r)}\tilde{f}(\sg(\tilde{v}))\dif x \\ & + \int_{\ball(z,s)\setminus\ball(z,t)}\tilde{f}(\sg(\tilde{v}+\rho(\tilde{u}-\tilde{v})))\dif x.
\end{align*}
We recall estimate \eqref{eq:usmall} and regroup terms to obtain 
\begin{align*}
\excenew(u;z,r) & \lesssim \int_{\ball(z,r)}\tilde{f}(\sg(\tilde{v}))\dif x +\int_{\ball(z,t)\setminus\ball(z,r)}\tilde{f}(\sg(\tilde{v}))-\tilde{f}(\sg(\tilde{u}))\dif x \\ & +\int_{\ball(z,s)\setminus\ball(z,t)}\tilde{f}(\sg(\tilde{v}+\rho(\tilde{u}-\tilde{v})))-\tilde{f}(\sg(\tilde{u}))\dif x =:\mathbf{I}+\mathbf{II}+\mathbf{III}.
\end{align*}
We now estimate the single terms. In view of $\mathbf{I}$, we directly work from the definition of $\tilde{f}$ and $\tilde{v}$. Then it follows by $\tilde{f}\approx\mathbf{e}$ and the definition of $\mathbf{m}'$ that  
\begin{align}
\mathbf{I} = \int_{\ball(z,r)}\tilde{f}(\sg(v)-\mathbf{m}')\dif x \lesssim \excenew(u;z,r). 
\end{align}
The term $\textbf{II}$ is already in a convenient form and shall be dealt with later. We turn to $\mathbf{III}$; recalling the product rule (i.e., $\sg(\varphi u)=\varphi\sg(u)+\nabla\varphi\odot u$ for $\varphi\colon\R^{n}\to\R$ and $u\colon\R^{n}\to\R^{n}$), monotonicity of $\mathbf{e}$ together with Lemma \ref{lem:eest}(i), we then have 
\begin{align*}
\int_{A(z;t,s)}\tilde{f}(\sg(\tilde{v}+\rho(\tilde{u}-\tilde{v})))\dif x &\lesssim \int_{A(z;t,s)}\mathbf{e}(\sg(\tilde{v}))+\mathbf{e}(\sg(\tilde{u}))+\mathbf{e}\left(\frac{\tilde{u}-\tilde{v}}{s-t}\right)\dif x.
\end{align*}
Keeping in mind $s<\tfrac{7}{8}R$ and therefore $s+\varepsilon+\delta < s+\tfrac{1}{8}R \leq R$, we see that $A(z;t-2\varepsilon,s+2\varepsilon)\subset \Omega$ and so it is admissible to estimate 
\begin{align*}
\int_{A(z;t,s)}\mathbf{e}(\sg(\tilde{v}))\dif x \leq \int_{A(z;t-2\varepsilon,s+2\varepsilon)}\mathbf{e}(\sg(\tilde{u}))\dif x \lesssim \int_{A(z;t-2\varepsilon,s+2\varepsilon)}\tilde{f}(\sg(\tilde{u}))\dif x. 
\end{align*}
Here we have used Jensen's inequality and implicitely the fact that $\delta=\varepsilon$. Combining the last inequality with Lemma \ref{lem:molest}, it follows that 
\begin{align*}
\int_{A(z;t,s)}\tilde{f}(\sg(\tilde{v}+\rho(\tilde{u}-\tilde{v})))\dif x &\lesssim \int_{A(z;t-2\varepsilon,s+2\varepsilon)}\tilde{f}(\sg(\tilde{u}))\dif x \\ & + \max\left\{\frac{\varepsilon^{2}}{(s-t)^{2}},\frac{\varepsilon^{p}}{(s-t)^{p}}\right\}\int_{A(z;t-2\varepsilon,s+2\varepsilon)}\tilde{f}(\sg(\tilde{u}))\dif x.
\end{align*}
Overall, we get the intermediate estimate 
\begin{align*} 
\mathbf{III} & \leq \int_{\ball(z,s)\setminus\ball(z,t)}\tilde{f}(\sg(\tilde{v}+\rho(\tilde{u}-\tilde{v})))+ \tilde{f}(\sg(\tilde{u}))\dif x \\
& \lesssim \int_{A(z;t-2\varepsilon,s+2\varepsilon)}\tilde{f}(\sg(\tilde{u}))\dif x + \max\left\{\frac{\varepsilon^{2}}{(s-t)^{2}},\frac{\varepsilon^{p}}{(s-t)^{p}}\right\}\int_{A(z;t-2\varepsilon,s+2\varepsilon)}\tilde{f}(\sg(\tilde{u}))\dif x
\end{align*}
and hence, putting together our preliminary estimates for $\mathbf{I},\mathbf{II}$ and $\mathbf{III}$, we end up with 
\begin{align}\label{eq:estml}
\begin{split}
\excenew(u;z,r) \lesssim \excenew(v;z,r) & + \int_{\ball(z,t)\setminus\ball(z,r)}\tilde{f}(\sg(\tilde{v}))-\tilde{f}(\sg(\tilde{u}))\dif x \\ & + \int_{A(z;t-2\varepsilon,s+2\varepsilon)}\tilde{f}(\sg(\tilde{u}))\dif x  \\ & + \max\left\{\frac{\varepsilon^{2}}{(s-t)^{2}},\frac{\varepsilon^{p}}{(s-t)^{p}}\right\}\int_{A(z;t-2\varepsilon,s+2\varepsilon)}\tilde{f}(\sg(\tilde{u}))\dif x
\end{split}
\end{align}
being valid for any $0<r<R/2$. The aim of the remaining proof is to utilize this inequality for certain choices of $s,t$. For $l\in\R_{0}^{+}$ denote the integer part of $l$ by $\lfloor l \rfloor$, i.e., 
\begin{align*}
\lfloor l \rfloor :=\max\{l'\in\mathbb{N}_{0}\colon\;l'\leq l\} 
\end{align*}
and define $K:=\lfloor 25(2\excesso^{\gamma/2}(u;z,R))^{-1}\rfloor$ with $0<\gamma<1/(n+2\alpha)$ given by Lemma \ref{lem:adjust}. For $k\in\{1,...,K\}$, put
\begin{align*}
a_{k}:=\frac{5}{8}R+k\frac{R}{400}\exce^{\gamma/2}(u;z,R)
\end{align*}
so that $a_{k}\in [\frac{5}{8}R,\frac{7}{8}R]$. Due to Lemma \ref{lem:convex}(ii), the choice $r<R/4$ implies that for each $k=1,...,K$ there exists $t_{k}\in(a_{8k-1},a_{8k})$ and $r_{k}\in(r,2r)$ with 
\begin{align*}
\int_{A(z;r_{k},t_{k})}\tilde{f}(\sg(\tilde{v}))-\tilde{f}(\sg(\tilde{u}))\dif x & \lesssim \varepsilon \max\left\{(a_{8k}-a_{8k-1})^{-1},r^{-1}\right\}\int_{\ball(z,R)}\tilde{f}(\sg(\tilde{u}))\dif x \\
& \lesssim \max\{\exce^{\gamma/2}(u;z,R),(R/r)\exce^{\gamma}(u;z,R)\}\times \\ 
& \times \int_{\ball(z,R)}\tilde{f}(\sg(\tilde{u}))\dif x. 
\end{align*}
We can choose $L>0$ so small such that for $k=1,...,K$ the annuli 
\begin{align*}
\mathcal{A}_{k}:=\ball(z,s_{k}+2\varepsilon)\setminus\ball(z,t_{k}-2\varepsilon),\;\;\;s_{k}:=t_{k}+LR\exce^{\gamma/2}(u;z,R)
\end{align*}
are pairwise disjoint and subsets of $\ball(z,R)$; for instance, the choice $L=\tfrac{1}{400}$ will do. We can therefore conclude that 
\begin{align*}
\int_{\mathcal{A}_{1}}\tilde{f}(\sg(\tilde{u}))\dif x +...+\int_{\mathcal{A}_{K}}\tilde{f}(\sg(\tilde{u}))\dif x \leq \int_{\ball(z,R)}\tilde{f}(\sg(\tilde{u}))\dif x
\end{align*}
and so we find $k'\in\{1,...,K\}$ such that 
\begin{align*}
\int_{\mathcal{A}_{k'}}\tilde{f}(\sg(\tilde{u}))\dif x \leq K^{-1}\int_{\ball(z,R)}\tilde{f}(\sg(\tilde{u}))\dif x \lesssim \exce^{\gamma/2}(u;z,R)\int_{\ball(z,R)}\tilde{f}(\sg(\tilde{u}))\dif x. 
\end{align*}
Going back to \eqref{eq:estml}, the choice $r=r_{k'}$, $t=t_{k'}$ and $s=s_{k'}$ now immediately yields 
\begin{align}\label{eq:esttwo}
\excenew(u;z,r)\lesssim \excenew(v;z,2r)+(1+R/r)\exce^{\gamma/2}(u;z,R)\int_{\ball(z,R)}\tilde{f}(\sg(\tilde{u}))\dif x 
\end{align}
for all $r<R/2$. This is not yet the decay estimate \eqref{eq:partial}, so we shall provide a suitable estimate on the last factor on the right side of Eq. \eqref{eq:esttwo}. By Lemma \ref{lem:eest}(i) we have that 
\begin{align*}
\int_{\ball(z,R)}\tilde{f}(\sg(\tilde{u})) &\lesssim \int_{\ball(z,R)}\mathbf{e}(\sg(u)-\mathbf{m})+R^{n}\mathbf{e}(\mathbf{m}-\mathbf{m}^{*})\\ 
&\lesssim \excenew(u;z,R) + (R/r)^{n}\int_{\ball(z,r)}\mathbf{e}((\sg(u)-\mathbf{m})_{\varepsilon,\delta})\\
& \lesssim (1+(R/r)^{n})\excenew(u;z,R), 
\end{align*}
the last estimate being valid due to $r+\varepsilon+\delta<R$. Combining this estimate with \eqref{eq:esttwo}, we obtain 
\begin{align}\label{eq:mainfinal}
\excenew(u;z,r)\lesssim \excenew(v;z,2r)+h(\excesso(u;z,R))\left(1+(R/r)^{n+1}\right)\excenew(u;z,R), 
\end{align}
we have set $h(t):=t^{\gamma/2}$. Since $h$ obviously is non--decreasing and moreover satisfies $\lim_{t\searrow 0}h(t)=0$, \eqref{eq:mainfinal} implies the claim. 
\end{proof}

\begin{proposition}\label{prop:propdecay}
In the situation of Proposition \ref{prop:main1} we have 
\begin{align*}
\devi(v;z,R/2)\lesssim h(\excesso(u;z,R))\excenew(u;z,R). 
\end{align*}
\end{proposition}
\begin{proof}
We fix $R/2<t<s<R$ and a constant $a>0$. We then put 
\begin{align*}
&\mathcal{A}:=\{w\in \sobo^{1,\infty}(\ball(z,t);\R^{n})\colon\; w = \tilde{v}\;\text{on}\;\partial\ball(z,t)\}\\
&\mathcal{B}:=\{w\in \sobo^{1,\infty}(\ball(z,s)\setminus\ball(z,t);\R^{n})\colon\; w = \tilde{v}\;\text{on}\;\partial\ball(z,t)\cup\partial\ball(z,s)\}
\end{align*}
and find $w_{1}\in\mathcal{A}$ and $w_{2}\in\mathcal{B}$ such that 
\begin{align*}
&\int_{\ball(z,t)}\tilde{f}(\sg(w_{1}))\dif x \leq \inf_{w\in\mathcal{A}}\int_{\ball(z,t)}\tilde{f}(\sg(w))\dif x + a \\
&\int_{\ball(z,s)\setminus\ball(z,t)}\tilde{f}(\sg(w_{1}))\dif x \leq \inf_{w\in\mathcal{B}}\int_{\ball(z,s)\setminus\ball(z,t)}\tilde{f}(\sg(w))\dif x + a. 
\end{align*}
Since $w_{1},w_{2}$ are Lipschitz and coincide on $\partial\ball(z,t)$, we deduce that $w_{3}:=\mathbbm{1}_{\overline{\ball(z,t)}}w_{1}+\mathbbm{1}_{\ball(z,s)\setminus\ball(z,t)}w_{2}$ belongs to $\sobo^{1,\infty}(\ball(z,t);\R^{n})$. Recalling the definition of the deviation functional $\devi$, we then obtain 
\begin{align*}
\devi(\tilde{v};z,t)&\leq \int_{\ball(z,t)}\tilde{f}(\sg(\tilde{v}))\dif x - \int_{\ball(z,t)}\tilde{f}(\sg(w_{1}))\dif x + a \\
&\leq \int_{\ball(z,s)}\tilde{f}(\sg(\tilde{v}))-\tilde{f}(\sg(\tilde{u}))\dif x - \int_{\ball(z,s)}\tilde{f}(\sg(w_{3}))-\tilde{f}(\sg(\tilde{u}))\dif x + 2a.
\end{align*}
Next, if we let $\rho$ be the function given by Eq. \eqref{eq:cutoff}, we obtain by exploiting minimality of $\tilde{u}$ that
\begin{align*}
\int_{\ball(z,s)}\tilde{f}(\sg(\tilde{u}))-\tilde{f}(\sg(w_{3}))\dif x & \leq \int_{\ball(z,s)}\tilde{f}(\sg(w_{3}+\eta(\tilde{u}-\tilde{v})))-\tilde{f}(\sg(w_{3}))\dif x \\
& \leq \int_{A(z;t,s)}\tilde{f}(\sg(w_{3}+\eta(\tilde{u}-\tilde{v})))-\tilde{f}(\sg(w_{3}))\dif x \\
& \lesssim \int_{A(z;t,s)}\mathbf{e}(\sg(w_{3}))+\mathbf{e}(\sg(\tilde{u}))+\mathbf{e}(\sg(\tilde{v}))\\ & +\mathbf{e}\left(\frac{\tilde{u}-\tilde{v}}{s-t}\right)\dif x. 
\end{align*}
Now note that $w_{3}=w_{2}$ on $\ball(z,s)\setminus\ball(z,t)$ and therefore
\begin{align*}
\int_{A(z;t,s)}\mathbf{e}(\sg(w_{3}))\dif x \lesssim \int_{A(z;t,s)}\tilde{f}(\sg(w_{3}))\dif x \lesssim \int_{A(z;t,s)}\tilde{f}(\sg(\tilde{v}))\dif x + a.
\end{align*}
Going back to the estimate of $\mathbf{III}$ in the proof of the preceding Proposition \ref{prop:main1}, we see that the same estimates used there also apply to the present problem, yielding 
\begin{align*}
\int_{\ball(z,s)}\tilde{f}(\sg(\tilde{u}))-\tilde{f}(\sg(w_{3}))\dif x \lesssim \int_{A(z,t-2\varepsilon,s+2\varepsilon)}\tilde{f}(\sg(\tilde{u}))\dif x + a.
\end{align*}
Because $a>0$ was assumed arbitrary and the constants which implicitely appear in the above estimates do not depend on $a$, we may pass $a\to 0$ in these estimates. In combination with our above preliminary estimate on $\devi(\tilde{v};z,R/2)$, we then obtain 
\begin{align*}
\devi(v;z,R/2) \lesssim  \int_{\ball(z,s)}\tilde{f}(\sg(\tilde{v}))-\tilde{f}(\sg(\tilde{u}))\dif x + \int_{A(z,t-2\varepsilon,s+2\varepsilon)}\tilde{f}(\sg(\tilde{u}))\dif x.
\end{align*}
We now view this last estimate as the substitute of Eq. \eqref{eq:estml} in the proof of Proposition \ref{prop:main1}. Working from this observation and employing the same strategy to find suitable choices of $s$ and $t$ as outlined in the proof of Proposition \ref{prop:main1}, the statement of the present proposition follows. 
\end{proof}
In the remaining part of this section we shall combine the preceding two propositions \ref{prop:main1} and \ref{prop:propdecay} with Lemma \ref{lem:adjust} and so obtain the desired decay estimate that will be the main ingredient for proving the partial regularity result. 
\begin{proposition}\label{prop:me}
There exists $\tau>0$ with the following property: If for $z\in\Omega$ and $R>0$ with $\ball(z,R)\subset\Omega$ a function $u\in X(\Omega)$ minimises 
\begin{align*}
\mathscr{F}_{R}[u]:=\int_{\ball(z,R)}f(\sg(u))\dif x 
\end{align*}
with its own boundary values and the excess satisfies $\excesso(u;z,R)<\tau$, then for any $0<r<R/2$ we have the estimate 
\begin{align}
\excesso(u;z,r)\lesssim ((r/R)^{n+2}+\vartheta(\excesso(u;z,R))(1+(R/r)^{n+1}))\excesso(u;z,R)
\end{align}
with a bounded an non--decreasing function $\vartheta\colon\R_{0}^{+}\to\R_{0}^{+}$ such that $\lim_{t\searrow 0}\vartheta(t)=0$. 
\end{proposition}
\begin{proof}
We may assume that $0<\tau<1$; the precise choice of $\tau$ will be determined throughout the proof. Let $u\in X(\Omega)$ be a local minimiser of $\mathscr{F}$. We denote $\mathbf{m}:=(\sg(u))_{z,R}$ and choose the mollification parameters $\delta$ and $\varepsilon$ as in Lemma \ref{lem:adjust} so that $\mathbf{t}(u_{\delta,\varepsilon};z,R/2)\leq c'\excesso^{\beta}(u;z,R)$, where $\beta>0$ is a constant provided by Lemma \ref{lem:adjust}. Hence, going back to Proposition \ref{prop:1} and taking $\tau$ so small such that $\excesso(u;z,R)<\tau$ implies $\mathbf{t}(u_{\delta,\varepsilon};z,R/2)<\min\{\sigma/c_{1},1\}$, we conclude that $v=u_{\delta,\varepsilon}$ satisfies 
\begin{align*}
\excenew(v;z,r) &\lesssim \left(\frac{r}{R}\right)^{n+2}\excenew(v;z,R/2)+\vartheta(\mathbf{t}(u;z,R/2))\int_{\ball(z,R/2)}|\sg(v)-\mathbf{m}|^{2}\dif x \\ & + \devi(v;z,R/2) &\\
&\lesssim ((r/R)^{n+2}+\vartheta(\mathbf{t}(u;z,R/2)))\excenew(u;z,R) + \devi(v;z,R/2), 
\end{align*}
where the second estimate is due to Corollary \ref{cor:adjust}. Diminishing $\tau$ further if necessary, we can assume in addition that $0<\tau<a$ with the smallness constant $a>0$ from Proposition \ref{prop:main1}. We then have, since $r<R/4$ implies that $2r<R/2$, 
\begin{align*}
\excenew(u;z,r) &\leq c\excenew(v;z,2r)+ch(\excesso(u;z,R))(1+(R/r)^{n+1})\excenew(u;z,R)\\
& \lesssim ((r/R)^{n+2}+\vartheta(\mathbf{t}(u;z,R/2)) + ch(\excesso(u;z,R))(1+(R/r)^{n+1}))\excenew(u;z,R) \\ & + \devi(v;z,R/2)\\
& \lesssim ((r/R)^{n+2}+\vartheta(\mathbf{t}(u;z,R/2)) + ch(\excesso(u;z,R))(1+(R/r)^{n+1}))\excenew(u;z,R) \\ & +  h(\excesso(u;z,R))\excenew(u;z,R), 
\end{align*} 
the last estimate being valid due to Proposition \ref{prop:propdecay}. Regrouping terms, we can immediately infer the claim. 
\end{proof}
\subsection{Iteration and Conclusion.}\label{sec:it}
In this section we shall show how an iteration of the estimates provided in Proposition \ref{prop:me} yields Theorem \ref{thm:1}. Regarding conditions (i)--(iii) in the following theorem, we see that they are clearly satisfied provided $f$ is a convex $\hold^{2}$--integrand having positive definite Hessian on the entire space $\R_{\sym}^{n\times n}$. Due to Lebesgue's Differentiation Theorem, condition (iii) will be satisfied at $\mathscr{L}^{n}$--any point $x_{0}\in\Omega$ so that the partial regularity assertion of Theorem  \ref{thm:1} will immediately follow. 
\begin{theorem}\label{thm:final}
Let $f\colon \R_{\sym}^{n\times n}\to\R$ be a convex function of $p$--growth, $1\leq p<\infty$ and let $u\in X(\Omega)$ be a local minimiser of the functional $\mathfrak{F}$ in $\Omega$. Let $x_{0}\in\Omega$ be and $\mathbf{m}\in \R_{\sym}^{n\times n}$ be such that 
\begin{itemize}
\item[\emph{(i)}] $f$ is of class $\hold^{2}(U)$ in an open neighbourhood $U$ of $\mathbf{m}$, 
\item[\emph{(ii)}] there exists $\lambda>0$ such that for all $\mathbf{Z}\in \R_{\sym}^{n\times n}$ we have 
\begin{align*}
\langle f''(\mathbf{m})\mathbf{Z},\mathbf{Z}\rangle\geq \lambda|\mathbf{Z}|^{2},
\end{align*}
\item[\emph{(iii)}] $\mathbf{m}$ is a Lebesgue point for $\sg(u)$ at $x_{0}$, that is, 
\begin{align*}
\lim_{r\to 0}\left(\dashint_{\ball(x_{0},r)}|\mathscr{E}u-\mathbf{m}|^{p}\dif x + \upsilon(p)\frac{|\E^{s}u|(\ball(x_{0},r))}{\mathscr{L}^{n}(\ball(x,r))}\right)= 0, 
\end{align*}
where $\upsilon(1)=1$ and $\upsilon(p)=0$ for $1<p<\infty$.
\end{itemize}
Then there exists a neighbourhood $\mathcal{O}$ of $x_{0}$ in which the minimiser $u$ is of class $\hold^{1,\alpha}$ for all $0<\alpha<1$. 
\end{theorem}
\begin{proof}
We follow closely the original proof in \cite{AG}. First we justify that (i)--(iii) imply the assumptions of Proposition \ref{prop:me}: Condition (i) ensures that we find $\sigma>0$ such that $f''$ is uniformly continuous on $\mathbb{B}(\mathbf{m},3\sigma)$. Therefore we find a modulus of continuity $\omega\colon\R_{0}^{+}\to\R_{0}^{+}$, i.e., a non--decreasing concave function with $\lim_{t\searrow 0}\omega(t)=0$, such that $|f''(\mathbf{m}')-f''(\mathbf{m}'')|\leq \omega(|\mathbf{m}'-\mathbf{m}''|)$ holds for all $\mathbf{m}',\mathbf{m}''\in\mathbb{B}(\mathbf{m},3\sigma)$. Referring to condition (i) and occasionally diminishing $\sigma$, we may further assume without loss of generality that there exist 
\begin{align*}
\lambda |\mathbf{Z}|^{2}\leq f''(\mathbf{m}')\mathbf{Z}\cdot\mathbf{Z}\leq\Lambda |\mathbf{Z}|^{2}
\end{align*}
holds for all $\mathbf{m}'\in\mathbb{B}(\mathbf{m},3\sigma)$ with two constants $0<\lambda\leq \Lambda \leq \infty$. Moreover, we see by the triangle inequality that these estimates also hold in $\mathbb{B}(\mathbf{m}',\sigma)$ for any $\mathbf{m}'\in\ball(\mathbf{m},2\sigma)$. In conclusion, Proposition \ref{prop:me} is available and provides us with $\tau>0$, a non--decreasing and bounded function $\vartheta\colon\R_{0}^{+}\to\R_{0}^{+}$ and a constant $c>0$ such that for all $0<r<R/2$ we have
\begin{align}\label{eq:iterative}
\excenew(u;z,r)\leq c((r/R)^{n+2}+\vartheta(\excesso(u;z,R))(1+(R/r)^{n+1}))\excenew(u;z,R)
\end{align}
provided $(\sg(u))_{x,R}=\mathbf{m}'$ and $\excesso(u;x,R)<\tau$. We shall use this to prove the following intermediate claim from which the assertion of the theorem is immediate: \\

\emph{Claim.} Let $0<\alpha<1$ be arbitrary. Then there exists $\tau'>0$ and $\tau_{0}>0$ such that the following holds: If $\ball(x,R)\subset\Omega$ and $\mathbf{m}'\in\mathbb{B}(\mathbf{m},\sigma)$ are such that $(\sg(u))_{x,R}=\mathbf{m}'$ and $\excesso(u;x,R)<\min\{\tau,\tau'\}$, then there holds $\excesso(u;x,\tau_{0}^{j}R)\leq \tau_{0}^{2\alpha j}\excesso(u;x,R)$ for all $j\in\mathbb{N}_{0}$. \\

We will verify the claim by induction on $j$. To do so, we put $\mathbf{m}_{j}:=(\sg(u))_{x,\tau_{0}^{k}R}$ and show that for all $j\in\mathbb{N}_{0}$ we have that 
\begin{align}\label{eq:induction}
|\mathbf{m}_{j}-\mathbf{m}'|< \sigma\sum_{k=0}^{j}2^{-k-1}\;\;\;\text{and}\;\;\;\excesso(u;x,\tau_{0}^{2\alpha j})\leq \tau_{0}^{2\alpha j}\excesso(u;x,R). 
\end{align}
Before embarking on the induction, we give a suitable candidate for $\tau_{0}$ and $\tau'$. From Lemma \ref{lem:eest} we know that there are two constants $C_{1},C_{2}>0$ such that $\mathbf{e}(\mathbf{A}+\mathbf{B})\leq C_{1}(\mathbf{e}(\mathbf{A})+\mathbf{e}(\mathbf{B}))$ for all $\mathbf{A},\mathbf{B}\in \R_{\sym}^{n\times n}$ and if $|\mathbf{A}|\leq 3$, then $\mathbf{e}(\mathbf{A})\leq C_{2}|\mathbf{A}|^{2}$. Recalling the constant $c>0$ from Eq. \eqref{eq:iterative}, we put 
\begin{align}
\tau_{0}:=\min\{2^{-\frac{1}{\alpha}},(2C_{1})^{-\frac{1}{2}},\frac{1}{2}\sigma(2C_{1}C_{2})^{-\frac{1}{2}},(2c)^{\frac{1}{2\alpha-2}}\}.
\end{align}
Having defined $\tau_{0}$, we pick $\tau'>0$ such that $\tau'<\tau_{0}^{n+2}$ and $0\leq t < \tau'$ implies $\vartheta(t)\leq \frac{1}{2}\tau_{0}^{2n+3}$, where $\vartheta$ is the function appearing in estimate \eqref{eq:iterative}. The latter property can be achieved because we have that $\lim_{t\searrow 0}\vartheta(t)=0$. 

We turn back to the induction. The intial step is certainly fulfilled. Now suppose that \eqref{eq:induction} holds for $j\in\mathbb{N}$. We record that in this case $|\mathbf{m}_{j}-\mathbf{m}'|<2\sigma$ holds as well and, since $\tau_{0}<1$, $\excesso(u;x,\tau_{0}^{j}R)\leq \tau$. We conclude that \eqref{eq:iterative} holds with $R$ being replaced by $\tau_{0}^{j}R$ and all $0<r<\tau_{0}^{j}R/2$. Now notice that with our choice of $\tau_{0}$ and $\tau'$, we also have that  $\excesso(u;x,\tau_{0}^{j+1}R)\leq \tau'$, $\tau_{0}^{-2-2n}\vartheta(\excesso(u;x,\tau_{0}^{j}R))\leq \frac{1}{2}$ and $\tau_{0}^{2-2\alpha}\leq 1/(2c)$. Therefore, using Eq. \eqref{eq:iterative}, we deduce that 
\begin{align*}
\excesso(u;v,\tau_{0}^{j+1}R) & = \tau_{0}^{-nj-n}R^{-n}\excenew(u;x,\tau_{0}^{j+1}R) \\
& \leq c\tau_{0}^{-n}(\tau_{0}^{n+2}+\vartheta(\excesso(u;x,\tau_{0}^{j}R))(1+\tau_{0}^{-n-1}))\excesso(u;x,\tau_{0}^{j}R)\\ 
& \leq c(\tau_{0}^{2}+\vartheta(\excesso(u;x,\tau_{0}^{j}R))(\tau_{0}^{-n}+\tau_{0}^{-2n-1}))\excesso(u;x,\tau_{0}^{j}R)\\ 
& \leq c\tau_{0}^{2-2\alpha}\tau_{0}^{2\alpha}(1+\vartheta(\excesso(u;x,\tau_{0}^{j}R))(\tau_{0}^{-n-2}+\tau_{0}^{-2n-3}))\excesso(u;x,\tau_{0}^{j}R)\\
& \leq c\tau_{0}^{2-2\alpha}\tau_{0}^{2\alpha}(1+\vartheta(\excesso(u;x,\tau_{0}^{j}R))(2\tau_{0}^{-2n-3}))\excesso(u;x,\tau_{0}^{j}R)\\ 
& \leq 2c\tau_{0}^{2-2\alpha}\tau_{0}^{2\alpha}\excesso(u;x,\tau_{0}^{j}R)\\ 
& \leq \tau_{0}^{2\alpha}\excesso(u;x,\tau_{0}^{j}R) \leq \tau_{0}^{2\alpha(j+1)}\excesso(u;x,R),
\end{align*}
where the last inequality holds by induction hypothesis. We now prove the remaining part of the claim; for this we abbreviate $\ball_{j}:=\ball(x,\tau_{0}^{j}R)$, $j\in\mathbb{N}$. We estimate, using Lemma \ref{lem:eest}(i) and performing a change of variables in the second step,
\begin{align*}
\mathbf{e}(\mathbf{m}_{j+1}-\mathbf{m}_{j}) & = \dashint_{\ball_{j+1}}\mathbf{e}(\mathbf{m}_{j+1}-\sg(u)(y)+\sg(u)(y)-\mathbf{m}_{j})\dif y \\
& \leq C_{1}\,\left(\dashint_{\ball_{j+1}}\mathbf{e}(\sg(u)(y)-\mathbf{m}_{j+1})\dif y + \tau_{0}^{-n}\dashint_{\ball_{j}}\mathbf{e}(\mathbf{m}_{j}-\sg(u)(y))\dif y\right)\\
& \leq C_{1}(\excesso(u;x,\tau_{0}^{j+1}R)+\tau_{0}^{-n}\excesso(u;x,\tau_{0}^{j}R))\\
& \leq C_{1}(\tau_{0}^{2\alpha(j+1)}+\tau_{0}^{-n}\tau_{0}^{2\alpha j})\excesso(u;x,R) \\
& \leq C_{1}\tau_{0}^{2\alpha j}(\tau_{0}^{2\alpha}+\tau_{0}^{-n})\excesso(u;x,R)\leq 2C_{1} \tau_{0}^{2\alpha j}\tau_{0}^{-n}\excesso(u;x,R).
\end{align*}
By our choice of $\tau_{0}$ and $\tau'$ we have $\excesso(u;x,R)\leq\tau'<\tau_{0}^{n+2}$ and $\tau_{0}^{2}\leq (2C_{1})^{-1}$. Therefore, 
\begin{align*}
2C_{1} \tau_{0}^{2\alpha j}\tau_{0}^{-n}\excesso(u;x,R)\leq 2C_{1} \tau_{0}^{2\alpha j+2}\leq \tau_{0}^{2\alpha j}
\end{align*}
and so we obtain $\mathbf{e}(\mathbf{m}_{j+1}-\mathbf{m}_{j})\leq 1$. But since $|\cdot|^{2}=(\mathbf{e}(\cdot)+1)^{2/p}-1$ we also get $|\mathbf{m}_{j+1}-\mathbf{m}_{j}|^{2}\leq 2^{2/p}-1\leq 9$ so that $|\mathbf{m}_{j+1}-\mathbf{m}_{j}|\leq 3$. Again recall that if $|\mathbf{A}|\leq 3$, then $\mathbf{e}(\mathbf{A})\leq C_{2}|\mathbf{A}|^{2}$ and thus 
\begin{align*}
|\mathbf{m}_{j+1}-\mathbf{m}_{j}|^{2} \leq C_{2}\mathbf{e}(\mathbf{m}_{j+1}-\mathbf{m}_{j})\leq 2C_{1}C_{2}\tau_{0}^{2\alpha j}\tau_{0}^{2}.
\end{align*} 
We observe that by our definition of $\tau_{0}$ there holds $2C_{1}C_{2}\tau_{0}^{2}\leq \sigma^{2}/4$ and $\tau_{0}<2^{-1/\alpha}$. Taking this into account, we then gain 
\begin{align*}
|\mathbf{m}_{j+1}-\mathbf{m}_{j}| \leq \frac{1}{2}\tau_{0}^{\alpha j}\sigma \leq 2^{-j-1}\sigma. 
\end{align*}
In conclusion, using the induction hypothesis we infer that 
\begin{align*}
|\mathbf{m}_{j+1}-\mathbf{m}_{j}|\leq |\mathbf{m}_{j+1}-\mathbf{m}|+|\mathbf{m}_{j}-\mathbf{m}|\leq \sigma\sum_{k=0}^{j+1}2^{-k-1}
\end{align*}
and thereby have proved the claim completely. \\

We shall now prove that the assumptions of the claim are satisfied, i.e., that for some $R>0$ with $\ball(x,R)\subset\Omega$ we have that $(\sg(u))_{x,R}=\mathbf{m}'$ and $\excesso(u;x,R)<\min\{\tau,\tau'\}$. For this observe that condition (iii) of the present theorem gives us 
\begin{align*}
\lim_{R\to 0}\dashint_{\ball(x_{0},R)}\sg(u)\dif x = \mathbf{m}\;\;\;\text{and}\;\;\;\lim_{R\to 0}\excesso(u;x_{0},R)=0
\end{align*}
and so we conclude by continuity of the maps $x\mapsto (\sg(u))_{x,r}$ and $x\mapsto \excesso(u;x,r)$, where $r>0$ is fixed, that  $(\sg(u))_{x,R}=\mathbf{m}'$ and $\excesso(u;x,R)<\min\{\tau,\tau'\}$ hold for some $\mathbf{m}'\in\mathbb{B}(\mathbf{m},\sigma)$ and a suitable $R>0$. In consequence, these estimates also remain valid in an open neighbourhood $\mathcal{O}$ of $x_{0}$. Now, since the assumptions of the claim are fulfilled, we are able to deduce that 
\begin{align}
\excesso(u;x,r)\leq C (r/R)^{2\alpha}\excesso(u;x,R)\;\;\;\text{for all}\;0<r<R, 
\end{align}
with a constant $C=C(n,N,\tau_{0},\alpha)>0$. In fact, let $0<r<R$ be arbitrary; then we find $j\in\mathbb{N}$ such that $\tau_{0}^{j+1}R<r\leq \tau_{0}^{j}R$ and thus, by the intermediate claim, 
\begin{align*}
\excesso(u;x,r) & \leq (\tau_{0}^{j}R/r)^{n}\excesso(u;x,\tau_{0}^{j}R) \lesssim (\tau_{0}^{j}R/r)^{n}\tau_{0}^{2\alpha j}\excesso(u;x,R)\\
& \leq \tau_{0}^{-n-2\alpha}\tau_{0}^{2\alpha (j+1)}\excesso(u;x,R) \leq C\,(r/R)^{2\alpha}\excesso(u;x,R),
\end{align*}
where $C=C(\alpha,\tau_{0})>0$ is a constant. Fixing a suitably small $R>0$, we have that $\excesso(u;x,r)\leq C r^{2}$ for all $x\in\mathcal{O}$ and $0<r<R$ with a constant $C=C(n,N,\tau_{0},\alpha,R)>0$. We finally refer to Lemma \ref{lem:eest}(ii) and thereby obtain the estimates 
\begin{align*}
&\dashint_{\ball(x,r)}|\sg(u)-(\sg(u))_{x,r}|^{p}\dif x \lesssim r^{\alpha p}\;\;\;\text{if}\;1\leq p\leq 2,\\
&\dashint_{\ball(x,r)}|\sg(u)-(\sg(u))_{x,r}|^{2}\dif x \lesssim r^{2\alpha }\;\;\;\text{if}\;2\leq p < \infty
\end{align*}
for all $x\in\mathcal{O}$ and all $0<r<R$. Since $\adop$ is elliptic, a standard application of Korn's Inequality yields the validity of the last two estimates with $\adop$ replaced by $D$. This establishes both Theorem \ref{thm:1} and \ref{thm:2}. 
\end{proof}
\begin{remark}
Interestingly, even the radially smmyetric the area--type integrands corresponding to $\m_{p}(t):=(1+|t|^{p})^{\frac{1}{p}}$, $t\in\R$, do not allow an application of theorem \ref{thm:1} if $p\neq 2$. This has been pointed out by \textsc{Schmidt} \cite{Schmidt} in the $\bv$--context and thus the above proof in the symmetric gradient context requires modification; indeed, if $p>2$, then $m''_{p}(0)=0$, and if $1<p<\infty$, the second derivative of $\m_{p}$ at the origin does not exist. We shall address the respective modifications of \textsc{Schmidt}'s work to the more general symmetric--gradient set up in a future publication.
\end{remark}

\end{document}